\documentclass[12pt]{amsart}

\usepackage[dvipsnames]{xcolor}

\usepackage{hyperref} 
\hypersetup{
    colorlinks=true,
    linkcolor=blue,
    citecolor=blue,
    urlcolor=blue
}

\usepackage{amsmath}
\usepackage{amssymb}
\usepackage{mathrsfs}
\usepackage{amscd}
\usepackage{graphicx}
\usepackage{mathdots}

\usepackage{gensymb}
\usepackage{epstopdf}
\usepackage{todonotes}
\usepackage{youngtab}
\usepackage{wrapfig}
\usepackage{float}

\usepackage{bbm}
\usepackage{bm}
\usepackage{verbatim}
\usepackage[margin=1.1in]{geometry}
\usepackage[nocompress]{cite}
\usepackage{epsfig}       
\usepackage{epic,eepic}   

\usepackage{tikz}
\usetikzlibrary{arrows.meta}

\usepackage{mathpazo}
\usepackage{domitian}
\usepackage[T1]{fontenc}

\usepackage{fancyhdr}
\newtheorem{thm}{Theorem}[section]

\newtheorem{proposition}[thm]{Proposition}

\theoremstyle{definition}
\newtheorem{definition}[thm]{Definition}

\theoremstyle{remark}

\numberwithin{equation}{section}

\begin{document}

\title{Behaviour near explosion\\ in continuous-time Galton--Watson trees}


\date{\today}

\author{Simon C. Harris}
\address{Department of Statistics, University of Auckland, Auckland, New Zealand}
\email{s.harris@auckland.ac.nz}

\author{Samuel G. G. Johnston}
\address{Department of Mathematics, King's College London, London, United Kingdom}
\email{samuel.g.johnston@kcl.ac.uk}

\author{Juan-Carlos Pardo}
\address{Centro de Investigaci\'on en Matem\'aticas (CIMAT), Guanajuato, Mexico}
\email{jcpardo@cimat.mx}

\begin{abstract}We study continuous-time Galton--Watson trees whose offspring generating function takes the form 
\begin{align} \label{eq:gen}
f(s) = s - (1-s)^\alpha L(1-s),
\end{align}
where $\alpha \in (0,1)$ and where $L:[0,1] \to [0,\infty)$ is slowly varying at zero. Processes with the generating function in \eqref{eq:gen} explode in finite time. We observe that, conditional on explosion at time $T$, at each  earlier time $t < T$ there is a unique particle $\xi_t$ alive at time $t$  who is an ancestor of all but finitely many particles at the explosion time. We call $(\xi_t)_{t \in [0,T)}$ the spine to explosion, and show that the births off the spine admit a Poissonian description, where they become both more frequent and larger as $t \uparrow T$.  

We undertake a careful study of the size of the population leading up to explosion,  showing in particular that near explosion, the rescaled population is approximately gamma distributed. More generally, conditional on explosion at time $T$, define a stochastic process $Z^\varepsilon := (Z_t^\varepsilon)_{t \in \mathbb{R}}$ by setting
\begin{align*}
Z_t^\varepsilon := E(\varepsilon e^{-t})N_{T-\varepsilon e^{-t}},
\qquad
E(t) = \mathbf{P}(N_t = \infty).
\end{align*}
 We show that as $\varepsilon \downarrow 0$, $Z^\varepsilon$  converges in finite-dimensional distributions to a stationary continuous-state branching process with immigration whose branching mechanism is subcritical and whose immigration mechanism corresponds to spine events. Finally, we show that this limiting process has a Markovian time reversal, and that the coalescent process associated with it is a stochastically time-changed Beta$(2-\alpha,\alpha)$-coalescent. 

\end{abstract}
\maketitle

\section{Introduction and overview}

\subsection{Explosions in branching processes}
A fundamental stochastic model for the evolution of a randomly evolving population is the continuous-time Galton--Watson tree.
In a continuous-time Galton--Watson tree with unit-rate lifetimes and offspring distribution $(p_n)_{ n \geq 0}$, we start at time zero with a single particle. This particle has a standard exponential lifetime, and at the end of this lifetime dies and is replaced by a random number $K$ of children, where $\mathbf{P}(K = n ) = p_n$. The children of the original particle repeat the behaviour of their parent independently of one another, i.e.\ they die at rate $1$ and are replaced by $n$ children upon death with probability $p_n$.

The overall behaviour of continuous-time Galton--Watson trees is well captured by the expectation $m = \sum_{ n \geq 0} n p_n$ of the offspring distribution, provided $m < \infty$. Indeed, if $(N_t)_{t \geq 0}$ is the stochastic process where $N_t$ counts the number of particles alive at time $t \geq 0$, then $\mathbb{E}[N_t] = e^{ (m-1) t}$. Moreover, various limit theorems (such as Yaglom's limit law and the Kesten-Stigum theorem) capture the asymptotic behaviour of these processes in the respective cases $m < 1, m = 1, m > 1$.

\begin{figure}[ht]
    \centering
    \includegraphics[width=0.8\textwidth]{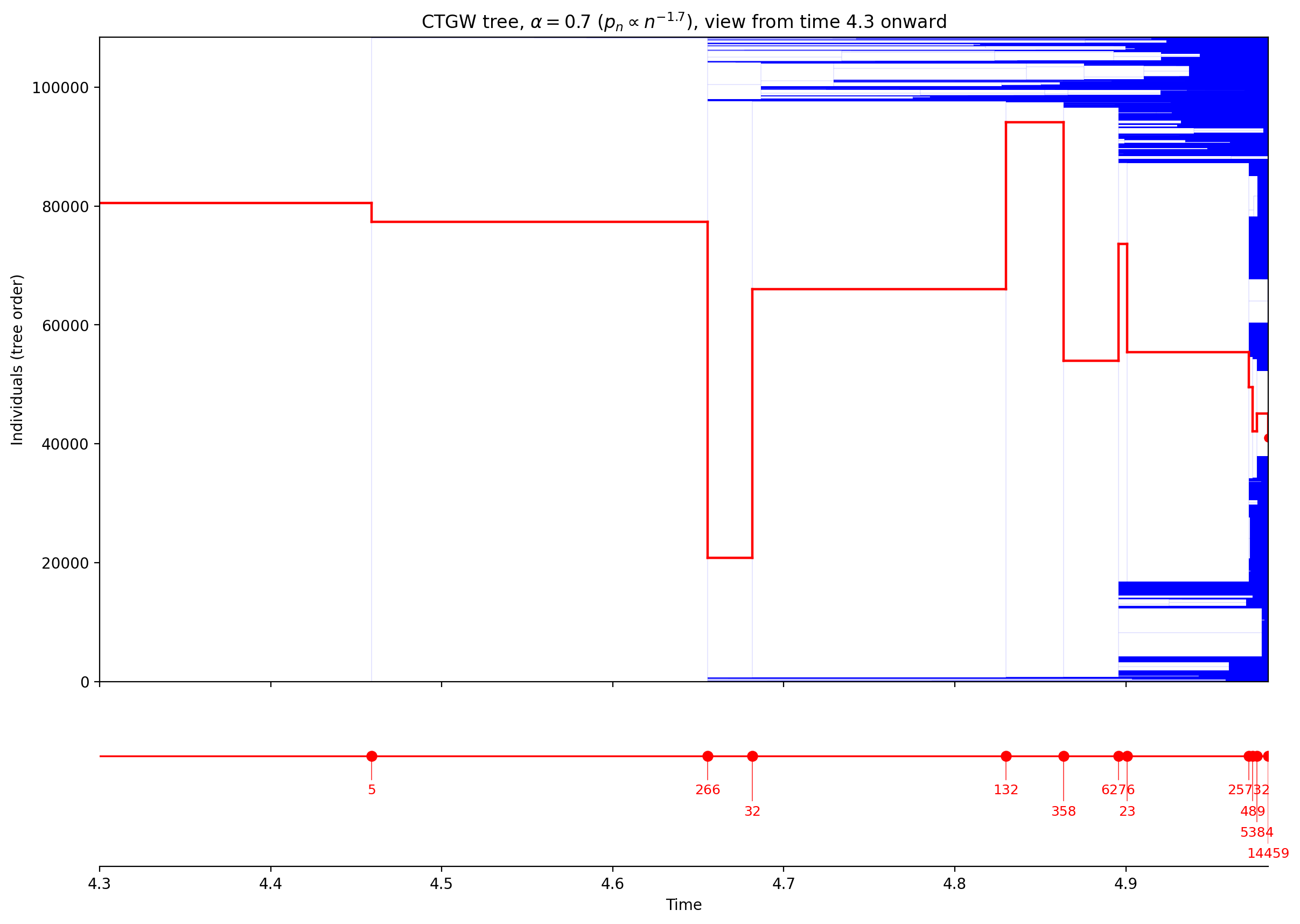}
    \caption{This figure depicts a simulation of a continuous-time Galton--Watson tree with offspring distribution $p_n = C/n^{1+\alpha}$ with $\alpha = 0.7$. In this simulation, the process is shown from time $4.3$ until time $T-\varepsilon := 4.98$, at which point the number of particles exceeds 100,000. The process explodes shortly after. The red lineage shows the lineage to explosion. The line below plots the moments at which the spinal ancestors died, and how many children they had upon death.}
    \label{fig:explosive}
\end{figure}

In the present article, we are interested in processes for which $m = \infty$, i.e.\ for which the expected number of children of any given particle in the process is infinite.
If the expectation of the offspring distribution satisfies $m = \infty$, then we also have $\mathbb{E}[N_t] = \infty$ for the expected number of particles alive at any time $ t> 0$. 
Continuous-time Galton--Watson trees with infinite expectation are relatively underexplored in the literature.

Firstly,
consider processes that `only just' have infinite mean. In particular, the process with offspring distribution $p_n = 1/(n(n-1))$ (for $n \geq 2$) has been studied in the literature, and it may be shown that 
its population size 
experiences double-exponential growth; in other words, its logarithm grows on the scale $e^t$ (see, e.g., Grey \cite{grey} for the analogue in the continuous-state case). 
Note that if $\mathbb{P}(K = n) = p_n$, then $\mathbf{E}[K] = \infty$ but $\mathbf{E}[K^{1-\varepsilon}]$ is finite for any $\varepsilon > 0$.

In the present paper, we will consider heavier tails still, where, for some $\alpha \in (0,1)$ the offspring distribution has a finite $\alpha-\varepsilon$ moment and infinite $\alpha+\varepsilon$ moment for all $\varepsilon > 0$. More specifically, we consider the case where the offspring generating function  $f(s) = \sum_{\ell \geq 0 }p_\ell s^\ell$  takes the form
\begin{align} \label{eq:offspring}
f(s) = s - (1-s)^\alpha L(1-s), \qquad \alpha \in (0,1),
\end{align}
where $L:[0,1]\to[0,\infty)$ is slowly varying at zero. Since $p_0=f(0)=-L(1)$ and both $p_0$ and $L(1)$ are nonnegative, this assumption forces $p_0=L(1)=0$, so extinction is excluded.
 
The simplest example of an offspring distribution with this form is
\begin{align*}
p_\ell = \frac{1}{\zeta(1+\alpha)}\frac{1}{\ell^{1+\alpha}}  \qquad \ell \geq 1,
\end{align*}
where $\zeta(s) := \sum_{n \geq 1} n^{-s}$ is the Riemann zeta function.

Continuous-time Galton--Watson trees with the offspring generating function given in \eqref{eq:offspring} have the striking property that they experience \emph{explosion in finite time}. In other words, there exists almost surely a random time $\tau \in (0,\infty)$, which we call the \textbf{explosion time}, such that
\begin{align*}
\text{$N_t < \infty $ for $t < \tau$} \qquad \text{and}  \qquad  \text{   $\lim_{t \uparrow \tau} N_t = \infty$}. 
\end{align*}
To avoid any ambiguity, we will also set $N_t = \infty$ for all 
$t \geq \tau$.
In this article we study in detail the behaviour of these continuous-time Galton--Watson trees as we approach explosion. In particular, we are interested in describing the population size and the relationships between different particles near explosion. 

\subsection{Goal of the present article}
The goal of the present article is to understand the behaviour of explosive continuous-time Galton--Watson trees with the offspring generating function in \eqref{eq:offspring} \emph{as we approach explosion}. More specifically, we would like to tackle questions such as:
\begin{enumerate}
\item \textbf{Explosion probability distribution.} Let $E(t) := \mathbf{P}(N_t = \infty)$ be the probability that the process has exploded by time $t$. That is, $E(t)$ is the cumulative distribution function of the law of the explosion time. What properties does $E(t)$ have? 
\item \textbf{Population size near explosion}. 
Suppose we condition on the event $\{ \tau = T \}$ that the process explodes at time $T > 0 $. 
What is the space-scale of the population before explosion? That is, conditional on explosion at time $T$, what is the order of the population size at time $T-\varepsilon$ as $\varepsilon \downarrow 0$? 
Can we prove a limit theorem for the law of the population size at time $T-\varepsilon$? 
\item \textbf{Population timescale near explosion}. What is the time-scale of the population before explosion? More explicitly, again conditioning on $\{\tau = T \}$, for fixed positive reals $u_1>u_2$, what is the dependence between the population size at time $T-u_1\varepsilon$ and time $T - u_2\varepsilon$? Is there a favourable time change we can use to understand the behaviour of the population size process as we approach explosion? How do our answers to this question and the previous one depend on the actual explosion time $T$? 
\item \textbf{Joint coalescent structure near explosion.} Finally, can we describe the coalescent structure of the tree at, or close to, explosion? Namely, if we pick $k$ particles uniformly at random from the population at time $T-\varepsilon$, what does the joint genealogical tree of these particles look like? On what time-scale did their most recent common ancestor last exist? 
\end{enumerate}

Indeed, we will answer all these questions in what follows. 
More precisely, we will show that under the right time and space scaling, the population converges in distribution to a continuous-state branching process with immigration, and we will describe the limiting genealogical structure by sampling directly from this limit process.

\subsection{Overview of our main results}
We briefly outline our main results.
First we discuss basic properties of the probability $E(t)$ of having exploded by time $t$. We are able to characterise both the small- and large-time behaviour of $E(t)$. We show in Proposition \ref{prop:largeE} that $1 - E(t)$ decays exponentially for \emph{large times} $t$. 
More importantly for us, in Proposition \ref{prop:smallE} we show that at \emph{small times} the explosion probability satisfies
\begin{align*}
E(t)
=
t^{\frac{1}{1-\alpha}}L_1(t),
\end{align*}
where $L_1(\cdot)$ is a slowly varying function at zero that can be described in terms of the slowly varying function $L$ in \eqref{eq:offspring}. 

Perhaps surprisingly, the small-$t$ asymptotics of $E(t)$ dictate the likely order of the population size near explosion. Indeed, suppose we would like to condition on explosion happening at time $T$, and at time $T-\varepsilon$ we have $N$ particles alive. Each of these $N$ particles has a small probability $E(\varepsilon)$ of giving rise to an explosion by time $T$. For exactly one of these events to occur with probability of order one, the natural scale is therefore $N$ of order $E(\varepsilon)^{-1}$. In fact, by renormalising the population by $E(\varepsilon)$, we can prove that conditional on explosion at time $T$, we have the limiting convergence in distribution:
\begin{align} \label{eq:fixedlaw}
E(\varepsilon)N_{T-\varepsilon}
\Longrightarrow
\mathrm{Gamma}(\alpha,1)
\qquad
\text{as }\varepsilon\downarrow0.
\end{align}
In particular, the limiting distribution does not depend on the value of the explosion time $T$. Theorem \ref{thm:main} strengthens this to finite-dimensional convergence under a simultaneous change of space and time.

Our perspective on these explosive continuous-time Galton--Watson trees is significantly enriched by the notion of what we call the \textbf{spine to explosion} $(\xi_t)_{t\in[0,T)}$. Conditional on explosion at time $T$, at each time $t<T$ there is a unique particle $\xi_t$ whose descendants are responsible for the explosion at time $T$. With this picture in mind, by making good use of the branching property we can decompose the population at time $t$ into $\xi_t$, the spine particle, and $N_t-1$ nonspine particles. 

The explosive spine particle has dramatically different behaviour from its nonspine relatives, giving birth at an increasingly rapid rate as we approach explosion, and to increasingly large population sizes.
In Theorem \ref{thm:spinethm} we describe this behaviour quantitatively. A small time $\varepsilon$ before explosion, the spine gives birth at asymptotic rate
\begin{align*}
\frac{\alpha}{1-\alpha}\frac{1}{\varepsilon},
\end{align*}
while the number of children produced at such a birth is of order $E(\varepsilon)^{-1}$. After multiplication by $E(\varepsilon)$, these birth sizes converge to a Gamma$(1-\alpha,1)$ distribution. Equivalently, after making the logarithmic time change
\begin{align*}
t\longmapsto T-\varepsilon e^{-t},
\end{align*}
the spinal branching events occur asymptotically at a constant rate. Thus, as the explosion time approaches, spinal births become both increasingly frequent and increasingly large.

This logarithmic time change also turns out to be the natural time scale for the entire population. Conditional on explosion at time $T$, define a stochastic process $Z^{\varepsilon,T} := (Z_t^{\varepsilon,T})_{t \in \mathbb{R}}$ by 
\begin{align*}
Z_t^{\varepsilon,T}
:=
E(\varepsilon e^{-t})
N_{T-\varepsilon e^{-t}},\qquad t\geq-\log(T/\varepsilon).
\end{align*}
For earlier $t$, set $Z_t^{\varepsilon,T}=E(T)$.
Our main population-size result, Theorem \ref{thm:main}, states that, as $\varepsilon\downarrow0$, the process $Z^{\varepsilon,T}$  converges in finite-dimensional distributions to a stationary Markov process $(Z_t)_{t\in\mathbb R}$ whose law is again independent of $T$. The limiting process is a continuous-state branching process with immigration. Its branching mechanism is subcritical, while its immigration mechanism consists of jumps arriving at rate $\alpha/(1-\alpha)$ with Gamma$(1-\alpha,1)$ distributed sizes. The stationary law is Gamma$(\alpha,1)$, and the immigration represents births from the spine to explosion. After rescaling, the limiting law depends only on $\alpha$, not on $T$ or the slowly varying function $L$.

Finally,  we study the genealogy naturally associated with the limiting process using the natural notions of genealogy for continuous-state branching processes \cite{BLG, JL}. Since $(Z_t)_{t\in\mathbb R}$ is stationary, it admits a right-continuous Markovian time reversal $(R_t)_{t\in\mathbb R}$, where $R_t=Z_{(-t)-}$. Forward-time branching and immigration jumps become downward jumps of the reversed population, and a jump which removes a proportion $v$ of the population gives rise naturally to a merger in which each ancestral lineage participates independently with probability $v$. The resulting jump measure has the form
\begin{align*}
\Gamma(1+\alpha)R_{t-}^{1-\alpha}v^{-2}
\mathcal B_{2-\alpha,\alpha}(\mathrm{d}v),
\end{align*}
where $\mathcal B_{2-\alpha,\alpha}$ is the Beta$(2-\alpha,\alpha)$ distribution and $R$ is the time-reversed population process. 
Consequently, after a random time change with instantaneous speed proportional to $R_{t-}^{1-\alpha}$, the genealogy is a standard Beta$(2-\alpha,\alpha)$-coalescent.

The spinal decomposition remains visible in this coalescent description. Conditional on a reverse jump which removes a proportion $v$ of the population, the jump arose from forward-time spinal immigration with probability $v$, and from a nonspine branching event with probability $1-v$. This leads to a two-type version of the limiting coalescent, in which one distinguished red block records those sampled lineages which have merged into the spine to explosion. 
After the same random time change, this coloured process may be realised particularly simply by taking a standard Beta$(2-\alpha,\alpha)$-coalescent on $k+1$ particles, declaring the block containing the $(k+1)$-st particle to be the spinal block, and then projecting onto the original $k$ sampled particles.

\subsection{Overview of this article}

The remainder of this article is structured as follows. 
We state our main results and discuss some related literature in Section \ref{sec:results}.
We then study in Section \ref{sec:foundations} the basic properties of our branching process at fixed times and fixed explosion times, and discuss the explosive spine decomposition.
In Section \ref{sec:A}, we prove our first main result, Theorem \ref{thm:spinethm}, on the asymptotic behaviour of the spine particle as we  approach explosion.
In Section \ref{sec:B}, we prove our second main result, Theorem \ref{thm:main}, which describes the overall population size process as we approach explosion.
In Section \ref{sec:C}, we prove our final main results, which discuss the time-reversed population size process and the associated coalescent process. 

\section{Main results and further discussion} \label{sec:results}

\subsection{Explosion time and explosion scale}
There is a natural scale on which to view the number of particles alive as we approach explosion, and it is intimately related to the small-$\varepsilon$ asymptotics of the probability that a population starting with a single particle has exploded by time $\varepsilon$. In this direction, 
recall
\begin{align*}
E(t) = \mathbf{P}( \tau \leq t) = \mathbf{P}( N_t = \infty)
\end{align*}
denotes the probability of explosion by time $t$ when starting with a single particle at time $0$. We are interested in the small-$\varepsilon$ asymptotics of $E(\varepsilon)$, and we will show that it takes the form 
\begin{align} \label{eq:smallE}
E(\varepsilon) = \varepsilon^{\frac{1}{1-\alpha}} L_1(\varepsilon),
\end{align}
where $L_1$ is a slowly varying function that can be characterised in terms of the slowly varying function $L$ in \eqref{eq:offspring}. 

Heuristically, a population of order $1/E(\varepsilon)$ has a probability of order one of exploding within the next $\varepsilon$ units of time. In other words $E(\varepsilon)$ is a natural factor by which to rescale the number of particles alive at a time $\varepsilon$ before explosion. Let us just say here that it is a consequence of our later results that, conditional on $\tau = T$, the law of the random variable $Z^\varepsilon_0 := E(\varepsilon)N_{T-\varepsilon}$ converges as $\varepsilon \downarrow 0$ to a $\mathrm{Gamma}(\alpha,1)$ random variable, in the sense that
\begin{align*}
\lim_{\varepsilon \downarrow 0 } \mathbf{P}( E(\varepsilon)N_{T-\varepsilon} \leq x | \tau = T ) = \frac{1}{\Gamma(\alpha)}\int_0^x u^{\alpha-1}e^{-u}\mathrm{d}u.
\end{align*}
We will refine this result substantially in the sequel.
In particular, we will consider the scaled population size as a \emph{process} as it approaches the explosion time.

\subsection{Main results}

\subsubsection{The spine process as we approach explosion}
Throughout, write $\mathbf P^T=\mathbf P(\,\cdot\mid\tau=T)$ for the law conditional on explosion at time $T$, and $\mathbf E^T$ for its expectation. The precise version of this conditional law is specified in Section \ref{sec:conditioning}.

Our first main result is an asymptotic description of the behaviour of the spine particle as we approach explosion.
First, we give a definition of the spine to explosion. We will give a more precise definition in Section \ref{sec:foundations} using the Ulam--Harris labelling of a continuous-time Galton--Watson tree, but the definition we provide here should prove sufficient.

Let $\tau$ be the random explosion time of the process starting with a single individual at time $0$.
We make the basic observation that $\tau$ takes values in $(0,\infty)$ and has a continuous density with respect to Lebesgue measure.
Suppose now that there are $n$ particles alive at some time $t>0$. Conditional on the process up to time $t$, the explosion time has the law of $t+\min_{1\leq i\leq n}\tau_i$, where $\tau_1,\ldots,\tau_n$ are independent random variables distributed like $\tau$.
Since each $\tau_i$ has a density with respect to Lebesgue measure, the minimum is uniquely attained by one of the $\tau_i$ almost surely.
Said in different terms, given a finite set $\mathcal{N}_t$ of particles alive at a certain time before explosion, there is a unique particle $\xi_t\in\mathcal{N}_t$ with the property that $\xi_t$ has infinitely many descendants alive at time $\tau$, whereas every $u\in\mathcal{N}_t\setminus\{\xi_t\}$ has only a finite number of descendants alive at time $\tau$.

We also note that $(\xi_t)_{t\in[0,\tau)}$ is indeed an ancestral line of descent, in the sense that
\begin{align*}
0\leq t<u<\tau \implies \xi_t\preceq\xi_u,
\end{align*}
where $\xi_t\preceq\xi_u$ denotes the fact that $\xi_t$ is a (weak) ancestor of $\xi_u$.

With this in mind, we define:

\begin{definition}[The spine to explosion]
Let $\tau$ denote the explosion time of the process. The \textbf{spine to explosion} $(\xi_t)_{t\in[0,\tau)}$ is a distinguished lineage of descent running through the process up to the explosion time and defined for $t\in[0,\tau)$ by
\begin{align*}
\xi_t:=\text{the unique particle at time $t$ with an infinite number of descendants at time $\tau$}.
\end{align*}
\end{definition}

We will be interested in viewing the process through an enriched filtration in which we observe not just the entire underlying Galton--Watson tree up to the current time, but also which particle is the spine to explosion. In particular, with this enriched view of the process, we can ask questions about how the spine to explosion behaves forwards in time.

Note that when we condition on the event $\{\tau=T\}$, the spine to explosion is a stochastic process $(\xi_t)_{t\in[0,T)}$ indexed by $t\in[0,T)$. We find that in this conditional setting, as we approach time $T$, particles along the spine give birth increasingly frequently, and the sizes of these births also grow.

This phenomenon is encompassed by Theorem \ref{thm:spinethm} below.

We give a somewhat informal statement here and prove it in Section \ref{sec:A}.

\begin{thm} \label{thm:spinethm}
Conditional on the event $\{\tau=  T\}$, that is, explosion occurring at time $T$:
\begin{enumerate}
\item A time $\varepsilon$ before explosion, the spine to explosion particle gives birth to new particles at rate $(1+o(1))(\alpha/(1-\alpha))\varepsilon^{-1}$ as $\varepsilon \downarrow 0$. More specifically, for $0 <  s<t$,
\begin{align*}
\lim_{\varepsilon \downarrow 0} \mathbf P^T( \text{spine gave birth $j$ times in $[T-t\varepsilon, T-s \varepsilon]$} ) = \lambda_{s,t}^j e^{ - \lambda_{s,t}}/j! \quad j \geq 0,
\end{align*} 
where
\begin{align*}
\lambda_{s,t} := \frac{\alpha}{1-\alpha}\int_{s \varepsilon}^{t \varepsilon} \frac{1}{u} \mathrm{d}u =\frac{\alpha}{1-\alpha} \int_s^t \frac{1}{u} \mathrm{d}u.
\end{align*}
\item \begingroup
Let $K_\varepsilon$ have the offspring law at a spinal birth with remaining time $\varepsilon$, namely
\begin{align*}
\mathbf P^T(K_\varepsilon=k)
=\frac{kp_k(1-E(\varepsilon))^{k-1}}{f'(1-E(\varepsilon))},\qquad k\geq1.
\end{align*}
Then the birth size is approximately $E(\varepsilon)^{-1}$ times a Gamma$(1-\alpha,1)$ random variable. More precisely, for $\theta\geq0$,
\begin{align*}
\lim_{\varepsilon\downarrow0}\mathbf E^T[e^{-\theta E(\varepsilon)K_\varepsilon}]
=(1+\theta)^{-(1-\alpha)}.
\end{align*}
\endgroup
\end{enumerate}
\end{thm}

Theorem \ref{thm:spinethm} says that the explosive spine has more and more rapid births as the explosion time approaches, and that these births tend to also grow in size as we approach explosion. 
Interestingly, note that the limiting behaviour is independent of the actual time of explosion $T$ and depends only on how long is left until explosion. 
The proof of Theorem \ref{thm:spinethm} is easier than that of our later results, and follows from a straightforward analysis of generating functions. Nonetheless, it is valuable to state this result first as it provides important intuition that
helps us to interpret our subsequent results. 

\subsubsection{The total population size as we approach explosion}

Our next result is a precise description of the population size process as we approach explosion. First of all, let us note that Theorem \ref{thm:spinethm} gives some indication of the correct space and time scale on which to view the population size. Indeed, part (2) of Theorem \ref{thm:spinethm} indicates that we should multiply all population sizes a time $\varepsilon$ before explosion by $E(\varepsilon)$. A useful time-change is slightly more delicate, but we observe that for $s < t$ in $\mathbb{R}$, the number of spine births in the time interval $[T- \varepsilon e^{-s}, T - \varepsilon e^{ - t} ]$ is approximately Poisson distributed with parameter $(\alpha/(1-\alpha))(t-s)$ as $\varepsilon \downarrow 0$. This suggests that, conditional on $\tau = T$, it will be useful to study the process under the time change $[-\log(T/\varepsilon),\infty) \mapsto [0,T)$ given by 
\begin{align*}
t \mapsto T - \varepsilon e^{-t}
\end{align*}
as $\varepsilon \downarrow 0$. Note that at time $T - \varepsilon e^{ - t}$, we should then rescale the population size by $E(\varepsilon e^{-t})$. All told, conditional on explosion at time $\tau = T$, this discussion leads us to consider the stochastic process $Z^{\varepsilon,T} := (Z^{\varepsilon,T}_t)_{t \in \mathbb{R}}$ defined by
\begin{align*}
Z^{\varepsilon,T}_t := E(\varepsilon e^{-t}) N_{T-\varepsilon e^{-t}},
\end{align*} 
where, if $t < t_0 := - \log(T/\varepsilon)$ we simply set $Z^{\varepsilon,T}_t := Z^{\varepsilon,T}_{t_0} = E(T)N_0 = E(T)$.

Our next main result says that regardless of the value of $T$, conditional on $\tau=T$, $Z^{\varepsilon,T}$  converges in finite-dimensional distributions as $\varepsilon \downarrow 0$ to a Markov process indexed by real time. 

\begin{thm} \label{thm:main}
For each fixed $T>0$, under $\mathbf P^T$, that is, conditional on explosion at $\tau=T$, the stochastic process
\begin{align*}
Z^{\varepsilon,T}:=(Z_t^{\varepsilon,T})_{t\in\mathbb R}
\end{align*}
converges in finite-dimensional distributions as $\varepsilon\downarrow0$ to a stationary $(0,\infty)$-valued Markov process $(Z_t)_{t\in\mathbb R}$ with generator
\begin{align*}
\mathcal L\phi(x)
=
\mathcal L_{\mathrm n}\phi(x)
+
\mathcal L_{\mathrm s}\phi(x),
\end{align*}
where the nonspine generator is
\begin{align} \label{eq:ngen}
\mathcal L_{\mathrm n}\phi(x)
:=
-\frac{x}{1-\alpha}\phi'(x)
+C_\alpha x\int_0^\infty
\bigl(\phi(x+u)-\phi(x)\bigr)
u^{-1-\alpha}e^{-u}\,\mathrm du,
\end{align}
and the spinal immigration generator is
\begin{align} \label{eq:sgen}
\mathcal L_{\mathrm s}\phi(x)
:=
C_\alpha\int_0^\infty
\bigl(\phi(x+u)-\phi(x)\bigr)
u^{-\alpha}e^{-u}\,\mathrm du,
\end{align}
with
\begin{align*}
C_\alpha:=\frac{\alpha}{\Gamma(2-\alpha)}.
\end{align*}
The stationary distribution of this process is Gamma$(\alpha,1)$. That is, for every $t\in\mathbb R$,
\begin{align*}
\mathbf P(Z_{t}\in\mathrm dz)
=
\frac{z^{\alpha-1}e^{-z}}{\Gamma(\alpha)}\,\mathrm dz.
\end{align*}
\end{thm}

\begin{figure}[ht]
    \centering
    \includegraphics[width=0.8\textwidth]{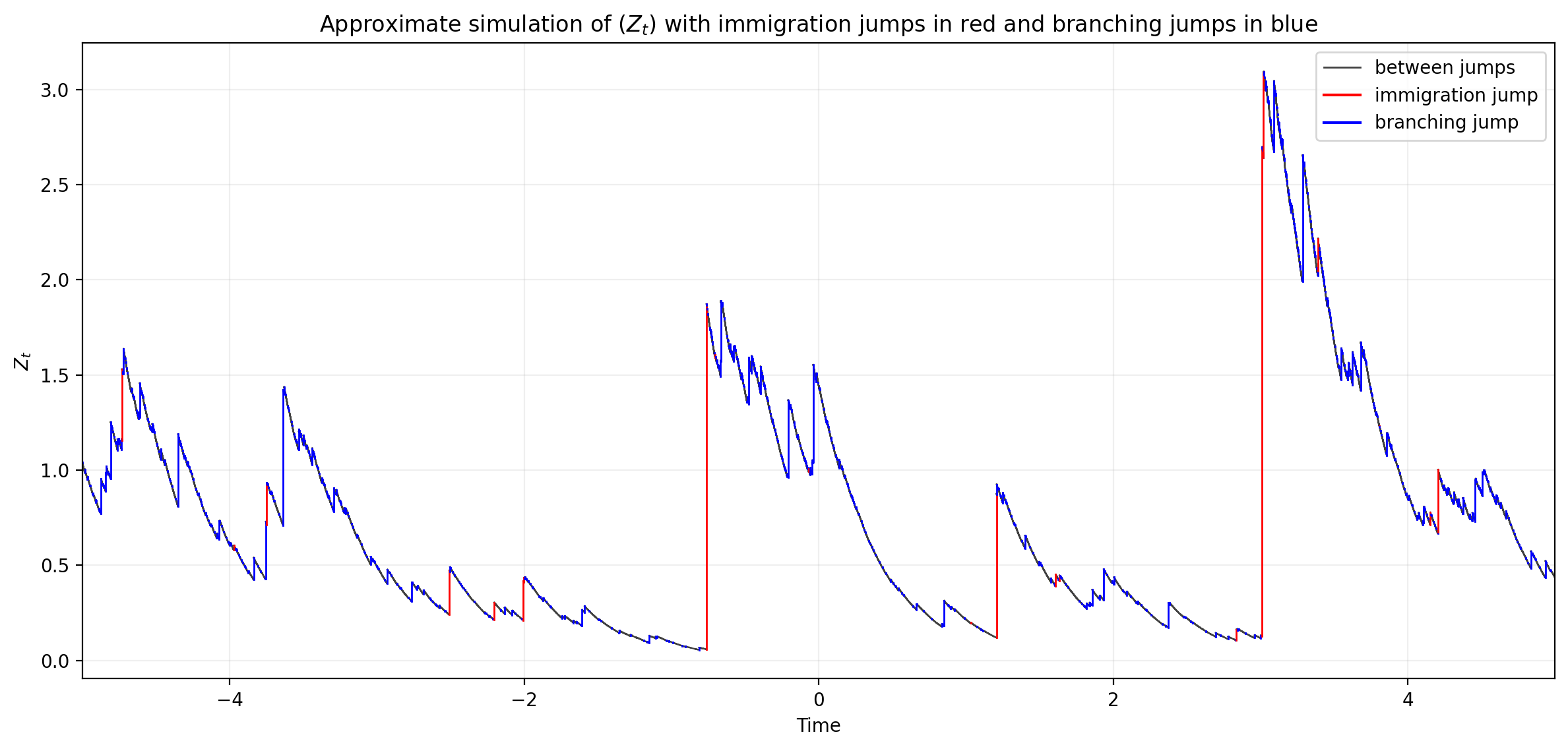}
    \caption{This figure depicts a simulation of the stochastic process $(Z_t)_{t \geq 0}$ in the case $\alpha = 0.7$. The jumps that are due to immigration along the explosive spine are illustrated in red.}
    \label{fig:Zt}
\end{figure}

Several remarks are in order. The process $(Z_t)_{t\geq0}$ is a continuous-state branching process with immigration. More specifically, its immigration and branching decompose as follows: 
\begin{itemize}
\item 
The immigration mechanism corresponds to asymptotic contributions to the population size from births along the spine itself. Its operator $\mathcal L_{\mathrm s}$ is the generator of a compound Poisson immigration process with jumps occurring at rate $\alpha/(1-\alpha)$ and jump sizes having the Gamma$(1-\alpha,1)$ distribution. 
\item
The branching mechanism describes the evolution of the large mass of nonspine particles leading up to explosion. Its
operator $\mathcal L_{\mathrm n}$ is the generator of a subcritical continuous-state branching process $(Y_t)_{t\geq0}$ satisfying
\begin{align*}
\mathbf E_x[e^{-\theta Y_t}]
=e^{-x u_t(\theta)},
\end{align*}
where
\begin{align} \label{eq:ut}
u_t(\theta)
=\left(1+e^{-t}\bigl((1+\theta)^{1-\alpha}-1\bigr)
\right)^{1/(1-\alpha)}-1.
\end{align}
This process is subcritical with unit-exponential decay, that is, $\mathbf E_x[Y_t]=xe^{-t}$.
In particular, $Y_t$ converges to zero almost surely as $t \to \infty$. However, let us note that $(Y_t)_{t \geq 0}$ \textit{extinguishes} (see, e.g.,\ \cite{LYZ}), i.e., it converges almost surely to zero as $t \to \infty$ without ever hitting zero in finite time. To see that it never hits zero, note that $\mathbf{P}_x(Y_t = 0) = \lim_{\theta \to \infty} \mathbf{E}_x[ e^{ - \theta Y_t} ] = 0$ since $\lim_{\theta \to \infty} u_t(\theta) = \infty$. 

\end{itemize}

Let us take a moment to observe that the population $Z_t$ is composed of mass that is descended from immigration at some time in the past, and that this fact encodes an identity in distribution between the immigration events and the stationary distribution. 

For concreteness, consider the population size $Z_0$ at time $0$. Now on the one hand, Theorem \ref{thm:main} says that $Z_0$ is $\mathrm{Gamma}(\alpha,1)$-distributed.
On the other hand, let us denote by $ \cdots < -T_i < \cdots < -T_2 < -T_1 < 0$ the times at which immigration events occurred before time zero. The structure of the immigration mechanism encodes the fact that the $-T_i$ are the points of a Poisson process of rate $\alpha/(1-\alpha)$ on $(-\infty,0]$, and that the sizes $(A_i)$ of the corresponding immigration events at time $-T_i$ are independent and identically distributed with a $\Gamma(1-\alpha,1)$ distribution. Let $B_i$ denote the total mass at time $0$ descended from the immigrants born at time $-T_i$. Then
\begin{align*}
B_i = Y^i_{T_i},
\end{align*}
where $(Y_t^i)_{t \geq 0}$ is an independent copy of $(Y_t)_{t \geq 0}$ starting from $Y_0^i = A_i$, and $(Y_t)_{t \geq 0}$ is the CSBP with Laplace exponent \eqref{eq:ut}. 
See Figure \ref{fig:campbell} for a depiction. 

\begin{figure}[ht]
    \centering
    \includegraphics[width=0.8\textwidth]{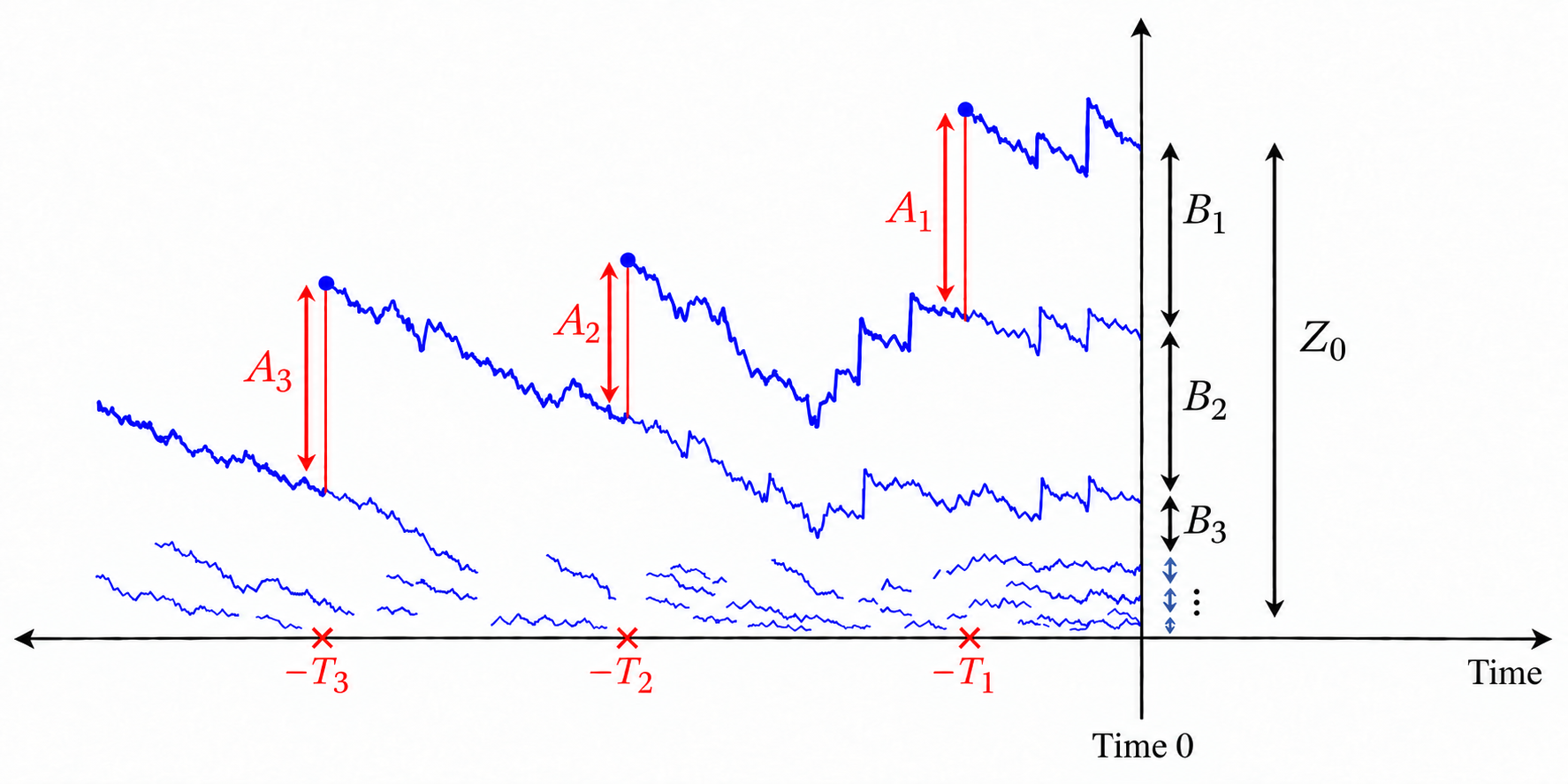}
    \caption{The points $(-T_i)_{i \geq 1}$ are the points of a Poisson process with rate $\alpha/(1-\alpha)$. At each event, mass of size $A_i$ immigrates to the population, where $(A_i)_{i \geq 1}$ are independent and identically distributed with law $\Gamma(1-\alpha,1)$. The ultimate contribution $B_i$ has the law of the time-$T_i$ size of a subcritical CSBP started with mass $A_i$.}
    \label{fig:campbell}
\end{figure}

We note that conditional on $(T_i,A_i)_{i \geq 1}$, the $(B_i)_{i \geq 1}$ are independent. Let us also highlight that since each copy of the subcritical CSBP $(Y_t)_{t \geq 0}$ is almost surely positive at all times $t \geq 0$, each of the $B_i$ is almost surely positive. 

The expected mass at time $0$ descended from the population present at time $-M$ is $\mathbf E[Z_{-M}]e^{-M}=\alpha e^{-M}$. These residual ancestral masses decrease as $M$ increases, so their limit is zero almost surely. Thus all the mass is descended from immigration at a finite time in the past, and we have the almost-sure identity
\begin{align} \label{eq:distid}
\sum_{i \geq 1} B_i = Z_0 \sim \mathrm{Gamma}(\alpha,1). 
\end{align}

The sum is almost surely finite, since
\begin{align*}
\mathbf E\left[\sum_{i\geq1}B_i\right]
=\frac{\alpha}{1-\alpha}\int_0^\infty (1-\alpha)e^{-t}\,\mathrm dt
=\alpha.
\end{align*}

In Section \ref{sec:campbell}, we give a proof of the distributional identity \eqref{eq:distid} directly using Campbell's formula. 

\medskip

Let us close this section by making a remark about our main results. In both Theorem \ref{thm:spinethm} and Theorem \ref{thm:main}, we conditioned on the event $\{ \tau = T \}$ that we have explosion at time $T$, and then proved some sort of convergence in distribution for an object (namely, the births off the spine to explosion and the overall population size) indexed by times of the order $T - O(\varepsilon)$ as $\varepsilon\downarrow 0$. Let us emphasise that in each case, we could have stated the theorem as a convergence in distribution for $\tau - O(\varepsilon)$ approaching the random time $\tau$. However, let us emphasise that as stated these results confirm something stronger: namely that the limiting behaviour of these objects a time $\varepsilon > 0$ before explosion as $\varepsilon \downarrow 0$ is independent of the actual explosion time itself.

\subsection{Time-reversal and the genealogy near explosion}

We now turn to the genealogical question posed in the introduction, which we address by sampling from the limiting CBI.

The direct approach to this question would involve sampling $k$ distinct particles at a time $\varepsilon$ before explosion, and then using spine techniques with exponential discounting to study their joint coalescent structure. More specifically, conditioning on explosion at time $T$ and then sampling $k$ particles at time $T-\varepsilon$ is equivalent to sampling $k$ particles from a population at time $T-\varepsilon$ but according to a process distributed under the change of measure 
\begin{align*}
\frac{\mathrm{d}\mathbb{Q}^{\varepsilon,T}}{\mathrm{d}\mathbb{P}}\Bigg|_{\mathcal{F}_{T-\varepsilon}} = 
\frac{N_{T-\varepsilon} (1-E(\varepsilon))^{N_{T-\varepsilon} - 1}}{\mathbf{E}[ N_{T-\varepsilon} (1-E(\varepsilon))^{N_{T-\varepsilon} - 1}] }.
\end{align*}
To sample $k$ distinct particles, one must further condition on $N_{T-\varepsilon}\geq k$; under $\mathbf P^T$ this event has probability tending to one as $\varepsilon\downarrow0$. The displayed density is interpreted as zero when $N_{T-\varepsilon}=\infty$.
Using a variant of \cite[Lemma 10]{HarrisJohnstonPardo2024}, it is possible to obtain explicit formulas for the genealogical process associated with sampling $k$ particles from the time $T-\varepsilon$ population of a process under $\mathbb{Q}^{\varepsilon,T}$. One can then study these formulas in a subsequent $\varepsilon \to 0$ scaling limit. 

In the present article, however, we will opt for a different, and cleaner, approach. We will sample the particles uniformly at random directly from the limiting process. That is, we take the sample from the limit, rather than the limit of the sample. We do not prove convergence of sampled genealogies from the discrete trees here. We instead use the machinery of genealogy for continuous-state branching processes.
Genealogies of continuous-state branching processes, including continuous-state branching processes with immigration (which we henceforth refer to as CBIs), have been described in several places \cite{Lambert2002,BLG,FoucartMaMallein2019}. 

We now describe the coalescent associated with the particular stationary CBI $(Z_t)_{t\in\mathbb R}$ in Theorem \ref{thm:main}, sampled at time $0$. Sample $k$ particles from the population at time $0$; equivalently, sample $k$ points uniformly from the space interval $[0,Z_0]$. 
We endow this CBI with ancestry by attaching every immigration family to a single immortal spinal lineage, which is not included in the population mass. Ancestors are traced along this spine once their immigration time is reached. We may then construct a coalescent process $(\pi_t)_{t \geq 0}$ (a stochastic process in the space of partitions of $\{1,\ldots,k\}$) 
by declaring 
\begin{align*}
\text{$i$ and $j$ are in the same block of $\pi_t$} \iff \text{same ancestor for $i,j$ in $Z_{-t}$ or on the spine}.
\end{align*}
In fact, we can enrich our understanding of this coalescent process by labelling a block red if it corresponds to a spine, or labelling it blue if it corresponds to a non-spine individual. Namely, we set 
\begin{align} \label{eq:block1}
\text{$i$ and $j$ are in a blue block of $\pi_t$} \iff \text{same ancestor for $i,j$ in $Z_{-t}$},
\end{align}
and
\begin{align} \label{eq:block2}
\text{$i$ and $j$ are in the red block of $\pi_t$} \iff \text{the spine is the time-$(-t)$ ancestor of $i$ and $j$}.
\end{align}

We use the standard branching genealogy, attaching each immigration family to the immortal spine as above; see \cite{Lambert2002,JL}. In this genealogy, at a branching jump each distinct nonspinal ancestral lineage belongs to the newly born mass independently with probability equal to its fraction of the post-jump population. The same sampling rule applies to the immigrant mass at an immigration jump, when participating lineages reach the spine. Our CBI has neither a Brownian component nor continuous immigration, so no additional continuous genealogical mechanism is needed. 
Given $Z$ let us define the time-reversed process
\begin{align*}
(R_t)_{t \in \mathbb{R}} := \text{right-continuous version of }(Z_{-t})_{t \in \mathbb{R}};
\end{align*}
that is, we set $R_t = Z_{(-t)-}$. 
Note that forward in time, $(Z_t)_{t \geq 0}$ has positive jumps. Thus backward in time, $(R_t)_{t \geq 0}$ has negative jumps at which it loses part of its mass.

We now give the following probabilistic construction of the coalescent process $(\pi_t)_{t \geq 0}$ from the reverse time process $(R_t)_{t \geq 0}$:

\begin{definition}[The coalescent process $(\pi_t)_{t\geq0}$] \label{def:coal}
Given the reverse-time population process $R=(R_t)_{t\geq0}$ run from time zero onwards, corresponding to the forward population process $Z$ on $(-\infty,0]$, define a coalescent process $(\pi_t)_{t\geq0}$ whose possible mergers can occur at the discontinuity times of $R$. If $t$ is a discontinuity time of $R$, at which $R_t$ loses a fraction $v_t$ of its total mass, i.e.:
\begin{align*}
v_t := \frac{R_{t-}-R_t}{R_{t-}},
\end{align*}
then, independently for each block of $\pi_{t-}$, mark that block with probability $v_t$. All marked blocks merge to form a single block. If fewer than two blocks are marked, the underlying partition does not change.
\end{definition}


This coalescent process is related to the so-called $\Lambda$-coalescent, which, given a finite measure $\Lambda$ on $[0,1]$, is the coalescent process with the property that when it has $k$ blocks, a particular subset of $j$ blocks merges at rate
\begin{align*}
\lambda_{k,j}
:=
\int_0^1v^{j-2}(1-v)^{k-j}\Lambda(\mathrm{d}v).
\end{align*}
The $\Lambda$-coalescent has the following interpretation: according to a Poisson process with rate $\Lambda(\mathrm{d}v)/v^2$, we generate a merger event with involvement probability $v$, at which every block in the process elects to be involved in a merger with probability $v$, or not involved with probability $1-v$. 


Generally speaking, the coalescent process arising from a CBI has a complicated distribution and does not have a simple time-change structure. It is however possible to describe the coalescent structure in terms of a $\Lambda$-coalescent with a merger measure that changes stochastically in time \cite{JL}.

We find that in our setting, there is a favourable decoupling of the population size with the merger measure, and that the coalescent process associated with the CBI $(Z_t)_{t \in \mathbb{R}}$ is a random time change of a Beta-coalescent. Indeed, we have the following result:

\begin{thm} \label{thm:coal}
Fix $k\geq2$, and let $(\pi_t)_{t\geq0}$ be the coalescent process 
corresponding to $R$ (as in Definition \ref{def:coal})
started from the partition of $\{1,\ldots,k\}$ into singletons. Define the random time change
\begin{align}\label{eq:timechange}
I(t)
:=
\inf \left\{
r \geq 0:
\int_0^r
\Gamma(1+\alpha)R_s^{1-\alpha}\,\mathrm{d}s
\geq t
\right\}.
\end{align}
Then $(\pi_{I(t)})_{t\geq0}$, when viewed in its own filtration, is a standard Beta$(2-\alpha,\alpha)$-coalescent. Equivalently, when the time-changed process has $b$ blocks, each prescribed collection of $j$ blocks merges at rate
\begin{align} \label{eq:lambda}
\lambda_{b,j}
=
\int_0^1
v^{j-2}(1-v)^{b-j}
\mathcal B_{2-\alpha,\alpha}(\mathrm{d}v),
\qquad
2\leq j\leq b,
\end{align}
where
\begin{align*}
 \mathcal B_{2-\alpha,\alpha}(\mathrm{d}v) = \frac{1}{\Gamma(\alpha)\Gamma(2-\alpha)}\,v^{1-\alpha}(1-v)^{\alpha-1} \mathrm{d}v, \qquad v \in [0,1],
\end{align*}
is the Beta distribution with parameters $2-\alpha$ and $\alpha$. 
\end{thm}

We prove Theorem \ref{thm:coal} by obtaining the transition rates of the time-reversed process, and showing that they have a tractable structure which behaves agreeably with the coalescent behaviour.

In fact, our understanding of the coalescent process is enriched significantly by understanding which discontinuities of $(R_t)_{t \in \mathbb{R}}$ were due to immigration, and which were due to branching.
With \eqref{eq:block1} and \eqref{eq:block2} in mind, we now have the following extension of Definition \ref{def:coal}, which creates a two-type coalescent process in which each block is either blue or red. It is convenient to regard this hidden spinal lineage as an additional red block (starting as $\{0\}$) which is always present, even when its intersection with $\{1,\ldots,k\}$ is empty.
See Figure \ref{fig:redblue}.

\begin{definition}[The two-type coalescent process $(\widetilde\pi_t)_{t\geq0}$]  \label{def:coal2}
Define the two-type coalescent process $(\widetilde\pi_t)_{t\geq0}$ on $\{0,1,\ldots,k\}$ associated with the reverse-time population process $R=(R_t)_{t\geq0}$ as follows. Initially, the sampled blocks $\{1\},\ldots,\{k\}$ are blue, and there is an additional 
red block $\{0\}$ which represents the explosive spine.
Suppose that $t$ is a discontinuity time of $R$ and set
\begin{align*}
v:=1-\frac{R_t}{R_{t-}}
\end{align*}
to represent the proportion of population mass removed at that time.
If the jump was due to immigration, colour it red, and if not, colour it blue. 
\begin{enumerate}
\item[(i)]
At a red jump, mark each blue block independently with probability $v$. 
The marked blue blocks merge together with the red block to form a new red block.
If no blue block is marked, there is no visible change.
\item[(ii)]
At a blue jump, mark each blue block independently with probability $v$. 
Any marked blue blocks merge together to form a new blue block. If
fewer than two are marked, the visible partition does not change. 
The red block is unaffected.
\end{enumerate}
\end{definition}

\begin{figure}[ht]
    \centering
    \includegraphics[width=0.5\textwidth]{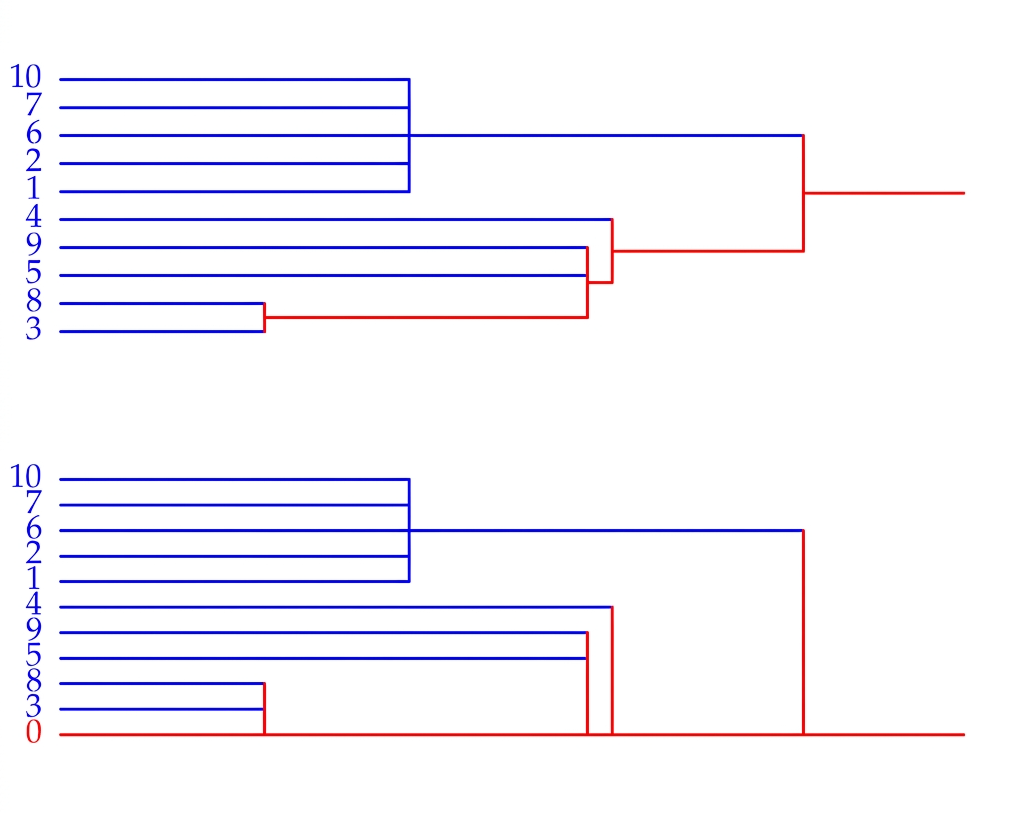}
    \caption{A depiction of the red/blue Beta-coalescent for $k=10$. In the upper panel the spine is hidden, while red sampled ancestry remains visible. In the lower panel the spine is visible.}
    \label{fig:redblue}
\end{figure}

We will realise the coloured coalescent process $(\tilde{\pi}_t)_{t \geq 0}$ as an example of a multitype $\Lambda$-coalescent \cite{JKR}. More specifically, consider the following coalescent process $(\bm \kappa_t)_{t\geq0} := ((\kappa_t,c_t))_{t \geq 0}$ which we call the red/blue coalescent and which is defined as follows. This process begins at time zero with a single red block and some blue blocks. Thereafter, at any moment where there are $b$ blue blocks, any prescribed $j$ of these $b$ blocks merge to form a single new blue block at rate
\begin{align*}
\lambda^{\mathrm{blue}}_{b,j} := \int_0^1 v^{j-2} (1-v)^{b-j} \left\{ (1-v) \mathcal{B}_{2-\alpha,\alpha}(\mathrm{d}v) \right\},\qquad 2\leq j\leq b,
\end{align*}
and any prescribed $j$ of these $b$ blocks merge with the red block to form a new red block at rate
\begin{align*}
\lambda^{\mathrm{red}}_{b,j} := \int_0^1 v^{j-2} (1-v)^{b-j} \left\{ v \mathcal{B}_{2-\alpha,\alpha}(\mathrm{d}v) \right\},\qquad 1\leq j\leq b.
\end{align*}
This red/blue coalescent can be realised as a standard $\mathrm{Beta}(2-\alpha,\alpha)$-coalescent on $\{0,1,\ldots,k\}$, with the block containing $0$ coloured red. Indeed, at an event with parameter $v$, the red block is marked with probability $v$, giving a red merger, and otherwise the marked blue blocks merge without it. Restricting to $\{1,\ldots,k\}$ and ignoring colours therefore gives a standard $\mathrm{Beta}(2-\alpha,\alpha)$-coalescent.

We are now ready to state our final main result, which 
generalises
Theorem \ref{thm:coal}.

\begin{thm} \label{thm:coal2}
Let $(\tilde{\pi}_t)_{t \geq 0}$ be as in Definition \ref{def:coal2} and let $(\bm \kappa_t)_{t \geq 0}$ denote a copy of the red/blue coalescent started from the red singleton $\{0\}$ and the blue singletons $\{1\},\ldots,\{k\}$. 
With the same time-change $I(t)$ as in Theorem \ref{thm:coal}, we have the identity in law
\begin{align*}
(\bm \kappa_t)_{t\geq0} \overset{d}{=} (\widetilde\pi_{I(t)})_{t\geq0}.
\end{align*}
\end{thm}

That completes the statements of our main results.


\subsection{Related work}

The literature on explosive continuous-time Galton--Watson processes is surprisingly sparse. The classical integral criterion for explosion was studied probabilistically by Schuh \cite{schuh} and Doney \cite{doney}, while Grey \cite{greyexplosion} showed that explosion cannot be characterised solely through the divergence of any prescribed moment of the offspring distribution. Pakes \cite{pakes} subsequently studied entrance laws, extensions and conditional limiting behaviour for explosive Markov branching processes. 

Let us touch on Sagitov's short article \cite{sagitov}, which is particularly relevant to our program: Suppose that, at some time, the branching process has a very large population $N$. Each of those $N$ particles independently has some small chance of generating an explosion soon. The process explodes when the first of those $N$ descendant families explodes. 

More precisely, Proposition 2 of \cite{sagitov} states, in our notation, that for $x\geq0$,
\begin{align*}
\mathbf P\bigl(\tau-t\leq E^{-1}(x/N)\mid N_t=N\bigr)\longrightarrow1-e^{-x}
\qquad\text{as }N\to\infty.
\end{align*}
Thus it is the transformed waiting time $N E(\tau-t)$ that has an asymptotically exponential law. Under our regular-variation assumption, ordinary rescaling by $a_N=E^{-1}(1/N)$ instead gives limiting survival function $\exp\{-x^{1/(1-\alpha)}\}$. Sagitov also identifies the corresponding regular-variation exponent for $E$. Our Gamma limit concerns the different conditioning in which the explosion time is fixed.

Related questions for explosive continuous-state branching processes, including quasi-stationary distributions and the associated (Q)-process, were considered by Labb\'e \cite{labbe}. A broader literature concerns explosion in age-dependent and Crump--Mode--Jagers branching processes, beginning with Grey \cite{greyagedependent}; see also Amini, Devroye, Griffiths and Olver \cite{amini} for the connection with min-summability in heavy-tailed branching random walks. 

More recent work of Iyer and Lodewijks \cite{iyerlodewijks} studies the genealogical trees associated with explosive Crump--Mode--Jagers branching processes, giving criteria for the tree at explosion to contain either an infinite path or a vertex of infinite degree, together with corresponding uniqueness results. Building on this work, Lodewijks \cite{lodewijks} obtains weaker sufficient conditions for the appearance of a unique vertex of infinite degree in explosive Crump--Mode--Jagers processes in random environments, with applications to super-linear preferential-attachment models.

\medskip
The coalescent structure of Galton--Watson trees is also an active field of research. 
Exact descriptions of finite-sample genealogies for uniform samples at a fixed time $T > 0$ for more general single- and multitype branching processes have also been obtained using spine and Poissonisation techniques; see, for example, \cite{Schweinsberg2003,HarrisJohnstonRoberts2020,JL,HarrisJohnstonPardo2024}.
Our setting is slightly unusual in that we study genealogical structure near explosion.

Connections between continuous-state branching processes and multiple-merger coalescents originate in the work of Bertoin and Le Gall \cite{BLG}, who related Neveu's branching process to the Bolthausen--Sznitman coalescent. Birkner et al.\ \cite{BirknerEtAl2005} subsequently showed that the genealogy of an $\alpha$-stable continuous-state branching process, after a random time change whose instantaneous speed is proportional to the $(1-\alpha)$-th power of the population size, is governed by the $\operatorname{Beta}(2-\alpha,\alpha)$-coalescent. Their result includes the non-conservative regime $0<\alpha<1$. Related extensions involving immigration, stationary branching populations, and time-reversed genealogies have been studied in \cite{FoucartHenard2013,AbrahamDelmasHe2021,FoucartMaMallein2019}; see also \cite{BerestyckiBerestyckiSchweinsberg2007} for the connection with stable trees.

\section{Fixed-time properties of the underlying continuous-time Galton--Watson tree} \label{sec:foundations}
\subsection{A word on slowly varying functions}

We begin with a quick word on slowly varying functions. We will refer to a measurable function $L:(0,\infty) \to [0,\infty)$, which is strictly positive in a neighbourhood of zero, as slowly varying at zero (or, for the remainder of the article, simply ``slowly varying''), if for all $c > 0$ we have
\begin{align*}
\lim_{t \downarrow 0} \frac{L(ct)}{L(t)} = 1.
\end{align*}
Further information on slowly varying functions may be found in \cite{BGT}.

\subsection{Generating functions and ODEs}

In this section we study the first properties of continuous-time Galton--Watson trees which give birth at unit rate and have offspring generating function
\begin{align*}
f(s) = \sum_{\ell \geq 0} p_\ell s^\ell
\end{align*}
taking the form
\begin{align} \label{eq:offspring2}
f(s) = s - (1-s)^\alpha L(1-s),
\qquad s \in [0,1),
\end{align}
for some $\alpha \in (0,1)$. We set $f(1)=1$.

Define the process generating function to be
\begin{align*}
F_t(s) := \mathbf{E}[s^{N_t}],
\qquad t \geq 0,\ s \in [0,1),
\end{align*}
where we adopt the convention that $s^\infty=0$ for $s \in [0,1)$. Thus, $F_t$ is the generating function for the total population size $N_t$ at time $t$, having started with a single individual alive at time $0$. By using the tower property, it is possible to show that the process generating function satisfies the semigroup property
\begin{align} \label{eq:semigroup}
F_t(F_{t'}(s)) = F_{t+t'}(s),
\qquad t,t' \geq 0,\ s \in [0,1).
\end{align}
The process generating function satisfies the differential equation
\begin{align} \label{eq:PDE}
\frac{\partial}{\partial t} F_t(s)
=
f(F_t(s)) - F_t(s),
\qquad
F_0(s)=s,
\end{align}
where $f(s)$ is the offspring generating function; see, for example, Athreya and Ney \cite{AN72}. When $f(s)$ takes the form in \eqref{eq:offspring2}, the process generating function has the property that
\begin{align*}
F_t(1-)
:=
\lim_{s \uparrow 1} F_t(s)
=
\mathbf{P}(N_t<\infty)
<1
\end{align*}
for $t>0$. In other words, the fact that explosion is possible is encoded in the differential equation in \eqref{eq:PDE}.

\subsection{The explosion probability $E(t)$}

In this section we look at properties of the explosion probability
\begin{align*}
E(t)
:=
\mathbf{P}(N_t=\infty)
=
1-F_t(1-).
\end{align*}
By taking $s \uparrow 1$ in \eqref{eq:semigroup}, one can show that
\begin{align} \label{eq:Esemigroup}
F_t(1-E(t'))=1-E(t+t').
\end{align}

Differentiating \eqref{eq:Esemigroup} with respect to $t'$ gives the useful identity
\begin{align} \label{eq:Fderivative}
F_t'(1-E(t'))
=
\frac{E'(t+t')}{E'(t')}.
\end{align}

Moreover, using \eqref{eq:PDE}, we see that $E(t)$ satisfies the ODE
\begin{align} \label{eq:EODE0}
 - E'(t)= f(1 -E(t)) - (1 - E(t)),
\qquad
E(0)=0,
\end{align}
which by \eqref{eq:offspring2} may instead be written 
\begin{align} \label{eq:EODE}
E'(t)=E(t)^\alpha L(E(t)),
\qquad
E(0)=0.
\end{align}
The explosion probability $E$ is the solution of this ODE which satisfies $E(t)>0$ for every $t>0$.

 Although we will not use the following proposition on the large-time behaviour of the explosion probability in the sequel, its proof is short, so we record it here:

\begin{proposition}[Large-$t$ behaviour of $E(t)$] \label{prop:largeE}
There is a constant $C\in(0,\infty)$ such that
\begin{align*}
1-E(t)\sim Ce^{-(1-p_1)t}
\qquad \text{as }t\to\infty.
\end{align*}
\end{proposition}

\begin{proof}
Write $v(t)=1-E(t)$. Then by \eqref{eq:EODE0} we have
\begin{align*}
v'(t)=f(v(t))-v(t).
\end{align*}
Since $f(s)=p_1s+O(s^2)$ as $s\downarrow0$ and $v(t)\downarrow0$ when $t\to\infty$, it follows that
\begin{align*}
\frac{v'(t)}{v(t)}
=
-(1-p_1)+O(v(t)).
\end{align*}

In particular, $v(t)$ decays at least exponentially fast for all sufficiently large $t$, and hence $\int_0^\infty v(t)\,\mathrm{d}t<\infty$. It follows that
\begin{align*}
\log v(t)+(1-p_1)t
\end{align*}
converges to a finite constant, and therefore $v(t)\sim Ce^{-(1-p_1)t}$ for some $C\in(0,\infty)$.

\end{proof}

\begin{proposition}[Small-$t$ behaviour of $E(t)$] \label{prop:smallE}
There exist slowly varying functions $L_1,L_2$ at zero such that
\begin{align} \label{eq:Eform}
E(t)
=
t^{\frac{1}{1-\alpha}}L_1(t),
\end{align}
and
\begin{align} \label{eq:E'form}
E'(t)
=
t^{\frac{\alpha}{1-\alpha}}L_2(t).
\end{align}
Moreover,
\begin{align} \label{eq:Ederivativeratio}
\lim_{t\downarrow0}
\frac{tE'(t)}{E(t)}
=
\frac{1}{1-\alpha}.
\end{align}
In particular, $E$ is regularly varying at zero with index $1/(1-\alpha)$, while $E'$ is regularly varying at zero with index $\alpha/(1-\alpha)$.
\end{proposition}

\begin{proof}
Recall from \eqref{eq:EODE} that
\begin{align*}
E'(t)
=
E(t)^\alpha L(E(t)).
\end{align*}
Separating variables gives
\begin{align*}
t
=
\int_0^{E(t)}
\frac{1}{y^\alpha L(y)}\,\mathrm{d}y.
\end{align*}
By Karamata's theorem, see \cite{BGT}, there is a slowly varying function $\widetilde L$ satisfying
\begin{align*}
\widetilde L(z)L(z)\to1
\qquad\text{as }z\downarrow0,
\end{align*}
such that
\begin{align*}
\int_0^z
\frac{1}{y^\alpha L(y)}\,\mathrm{d}y
\sim
\frac{1}{1-\alpha}
z^{1-\alpha}\widetilde L(z).
\end{align*}
Consequently,
\begin{align*}
t
\sim
\frac{1}{1-\alpha}
E(t)^{1-\alpha}\widetilde L(E(t)).
\end{align*}
The inversion theorem for regularly varying functions, again see \cite{BGT}, therefore implies that
\begin{align*}
E(t)
=
t^{\frac{1}{1-\alpha}}L_1(t)
\end{align*}
for some slowly varying $L_1$.

Using \eqref{eq:EODE},
\begin{align*}
E'(t)=E(t)^\alpha L(E(t)).
\end{align*}
Since $E$ is regularly varying at zero with index $1/(1-\alpha)$ and $L$ is slowly varying, $L(E(t))$ is slowly varying at zero. Hence
\begin{align*}
E'(t)
=
t^{\frac{\alpha}{1-\alpha}}L_2(t)
\end{align*}
for some slowly varying function $L_2$. Moreover, using the preceding asymptotic relation together with $\widetilde L(z)L(z)\to1$,
\begin{align*}
\frac{tE'(t)}{E(t)}
=
tE(t)^{\alpha-1}L(E(t))
\longrightarrow
\frac{1}{1-\alpha}.
\end{align*}
Equivalently,
\begin{align*}
\frac{L_2(t)}{L_1(t)}
\longrightarrow
\frac{1}{1-\alpha},
\end{align*}
which proves the result.

\end{proof}

\subsection{Conditioning on explosion to occur at time $T$} \label{sec:conditioning}

We now consider conditioning on explosion at a time $T$. All conditional distributions given $\{\tau=T\}$ below are understood in terms of a regular conditional distribution with respect to the density $E'(T)$ of the explosion time. For each $T>0$, we use the version specified on $\mathcal F_t$, $t<T$, by the density $N_t(1-E(T-t))^{N_t-1}E'(T-t)/E'(T)$, interpreted as zero when $N_t=\infty$. The semigroup identity makes these restrictions consistent.
The conditional behaviour of the process admits a particularly tractable description thanks to the branching property. Namely, by the branching property, for $0\leq t<T$ we have
\begin{align*}
\mathbf{P}(\tau>T\mid N_t)
=
\mathbf{P}(N_T<\infty\mid N_t)
=
(1-E(T-t))^{N_t}.
\end{align*}
In particular, by differentiating with respect to $T$, we see that the conditional density of the explosion time given a finite population size $N_t$ at time $t$ is
\begin{align*}
\mathbf{P}(\tau\in\mathrm{d}T\mid N_t)
&=
-\frac{\partial}{\partial T}
\mathbf{P}(\tau>T\mid N_t)\,\mathrm{d}T
\\
&=
E'(T-t)N_t
(1-E(T-t))^{N_t-1}\,\mathrm{d}T.
\end{align*}
It follows that
\begin{align} \label{eq:follows}
\mathbf{E}[s^{N_t} \mid\tau=T]
&=
\frac{1}{E'(T)}
\mathbf{E}\left[
s^{N_t}E'(T-t)N_t
(1-E(T-t))^{N_t-1}
\right] \nonumber 
\\
&=
\frac{sE'(T-t)}{E'(T)}
F_t'\bigl(s(1-E(T-t))\bigr).
\end{align}
More generally, let $0 \leq t \leq t' < T$. Then
\begin{align} \label{eq:polar0}
\mathbf{E} \left[ s^{N_{t'}} \big| N_t=n,\tau=T \right]
= \frac{ \mathbf{E} \left[ s^{N_{t'}} \mathrm{1}\{ \tau \in \mathrm{d}T \} \big| N_t=n \right]}{\mathbf{P} \left( \tau \in \mathrm{d}T \big| N_t=n \right)}.
\end{align}
Now the denominator above takes the form \begin{align} \label{eq:polar1}
\mathbf{P} \left( \tau \in \mathrm{d}T \big| N_t=n \right) = n (1-E(T-t))^{n-1}E'(T-t) \mathrm{d}T.
\end{align} 
As for the numerator, using the tower property to obtain the first equality below, then using \eqref{eq:polar1} with $t'$ in place of $t$ to obtain the second, we have 
\begin{align} \label{eq:polar2}
\mathbf{E} \left[ s^{N_{t'}} \mathrm{1}\{ \tau \in \mathrm{d}T \} \big| N_t=n \right] &= \mathbf{E} \left[ s^{N_{t'}} \mathbf{P}( \tau \in \mathrm{d}T | N_{t'} ) \big| N_t=n \right] \nonumber \\
&= E'(T-t')\mathbf{E} \left[ s^{N_{t'}} N_{t'}(1-E(T-t'))^{N_{t'}-1}  \big| N_t=n \right]\mathrm{d}T.
\end{align}
Now $\mathbf{E}[ N_{t'} v^{N_{t'}-1} | N_t = n ] = \frac{\partial}{\partial v} \mathbf{E}[ v^{N_{t'}} | N_t = n ] = \frac{\partial}{\partial v} F_{t'-t}(v)^n = n F_{t'-t}'(v) F_{t'-t}(v)^{n-1}$. In particular, setting $z = s(1-E(T-t'))$ in \eqref{eq:polar2} we obtain
\begin{align} \label{eq:polar3}
\mathbf{E} \left[ s^{N_{t'}} \mathrm{1}\{ \tau \in \mathrm{d}T \} \big| N_t=n \right] &= n s E'(T-t') F'_{t'-t}( z ) F_{t'-t} ( z )^{n-1}\mathrm{d}T.
\end{align}
Using \eqref{eq:polar3} and \eqref{eq:polar1} in \eqref{eq:polar0} we obtain
\begin{align} \label{eq:conditionaltransition}
&\mathbf{E}\left[
s^{N_{t'}}
\,\middle|\, N_t=n,\tau=T
\right] =
\frac{sE'(T-t')}{E'(T-t)}
F_{t'-t}'\left(z\right)
\left(
\frac{
F_{t'-t}\left(z\right)
}{
1-E(T-t)
}
\right)^{n-1},
\end{align}
where $z = z_{s,T-t'} := s(1 - E(T-t'))$.


Setting $t=T-\varepsilon$ in \eqref{eq:follows} we have
\begin{align} \label{eq:pregen}
\mathbf{E}\left[
s^{N_{T-\varepsilon}}
\,\middle|\, \tau=T
\right]
=
\frac{sE'(\varepsilon)}{E'(T)}
F_{T-\varepsilon}'\bigl(s(1-E(\varepsilon))\bigr).
\end{align}

At this point we can already study the asymptotic behaviour of the population size at a fixed time approaching explosion. We now prove the following result.

\begin{proposition} \label{prop:size}
For every $x\geq0$, we have
\begin{align*}
\lim_{\varepsilon\downarrow0}
\mathbf{P}\left(
E(\varepsilon)N_{T-\varepsilon}\leq x
\mid \tau=T
\right)
=
\frac{1}{\Gamma(\alpha)} 
\int_0^x u^{\alpha-1}e^{-u}\,\mathrm{d}u.
\end{align*}
\end{proposition}

\begin{proof}
By standard results on convergence in distribution, it is sufficient to establish that, for all $\theta>0$,
\begin{align} \label{eq:LaplaceGoal}
\lim_{\varepsilon\downarrow0}
\mathbf{E}\left[
e^{-\theta E(\varepsilon)N_{T-\varepsilon}}
\mid \tau=T
\right]
=
(1+\theta)^{-\alpha}.
\end{align}

Let
\begin{align*}
s_\varepsilon
:=
e^{-\theta E(\varepsilon)}.
\end{align*}
Then, with the use of \eqref{eq:pregen} in mind, note that
\begin{align}
s_\varepsilon(1-E(\varepsilon))
=
1-(1+\theta)E(\varepsilon)
+O(E(\varepsilon)^2).
\end{align}
Since $E$ is continuous and strictly increasing and regularly varying at $0$, we  choose $v_{\theta,\varepsilon}>0$ such that
\[
E(v_{\theta,\varepsilon}\varepsilon)=(1+\theta)E(\varepsilon)
+O(E(\varepsilon)^2).
\]
In view of the expansion above, we specify this choice more precisely by

\begin{align} \label{eq:defv}
s_\varepsilon(1-E(\varepsilon))
=
1-E(v_{\theta,\varepsilon}\varepsilon).
\end{align}
It follows that
\begin{align*}
\frac{E(v_{\theta,\varepsilon}\varepsilon)}
{E(\varepsilon)}
= \frac{1-s_\varepsilon(1-E(\varepsilon))
}{E(\varepsilon)}
=
1+\theta+O(E(\varepsilon)).
\end{align*}
Since $E$ is regularly varying at zero with index $1/(1-\alpha)$, the uniform convergence and inversion properties of regularly varying functions imply that
\begin{align} \label{eq:vlimit}
v_{\theta,\varepsilon}
\longrightarrow
(1+\theta)^{1-\alpha}.
\end{align}

Using \eqref{eq:pregen} and \eqref{eq:defv}, we obtain
\begin{align*}
&\mathbf{E}\left[
e^{-\theta E(\varepsilon)N_{T-\varepsilon}}
\mid \tau=T
\right]
=
\frac{s_\varepsilon E'(\varepsilon)}{E'(T)}
F_{T-\varepsilon}'
\left(
1-E(v_{\theta,\varepsilon}\varepsilon)
\right).
\end{align*}
Applying \eqref{eq:Fderivative} with
$t = T- \varepsilon$ and $t'=v_{\theta,\varepsilon}\varepsilon$
gives
\begin{align*}
F_{T-\varepsilon}'
\left(
1-E(v_{\theta,\varepsilon}\varepsilon)
\right)
=
\frac{
E'\left(
T+(v_{\theta,\varepsilon}-1)\varepsilon
\right)
}{
E'(v_{\theta,\varepsilon}\varepsilon)
}.
\end{align*}
Consequently,
\begin{align} \label{eq:LaplaceFactorisation}
&\mathbf{E}\left[
e^{-\theta E(\varepsilon)N_{T-\varepsilon}}
\mid \tau=T
\right]=
s_\varepsilon
\frac{E'(\varepsilon)}
{E'(v_{\theta,\varepsilon}\varepsilon)}
\frac{E'\left(
T+(v_{\theta,\varepsilon}-1)\varepsilon
\right)}{E'(T)}.
\end{align}

As $\varepsilon\downarrow0$, we have $s_\varepsilon\to1$. Moreover, by \eqref{eq:vlimit} and the continuity of $E'$ at $T>0$,
\begin{align*}
\frac{
E'\left(
T+(v_{\theta,\varepsilon}-1)\varepsilon
\right)
}{
E'(T)
}
\longrightarrow1.
\end{align*}
Finally, by \eqref{eq:E'form}, the function $E'$ is regularly varying at zero with index $\alpha/(1-\alpha)$. Hence, using \eqref{eq:vlimit},
\begin{align*}
\frac{E'(\varepsilon)}
{E'(v_{\theta,\varepsilon}\varepsilon)}
\longrightarrow
\left((1+\theta)^{1-\alpha}\right)^{-\alpha/(1-\alpha)}
=
(1+\theta)^{-\alpha}.
\end{align*}
Substituting these limits into \eqref{eq:LaplaceFactorisation} proves \eqref{eq:LaplaceGoal}. Since $(1+\theta)^{-\alpha}$ is the Laplace transform of a Gamma random variable with shape parameter $\alpha$ and rate parameter $1$, the result follows.
\end{proof}

\subsection{The spine to explosion}

In this section we introduce the spine to explosion $(\xi_t)_{t \in [0,\tau)}$. We make this definition precise using the Ulam--Harris labelling. Let
\begin{align*}
\mathbb{T}
:=
\{\varnothing\}
\cup
\bigcup_{n \geq 1}\mathbb{N}^n,
\end{align*}
where $\varnothing$ denotes the empty word. We associate every individual in a continuous-time Galton--Watson tree starting at time zero with a label in $\mathbb{T}$ as follows. The initial individual has label $\varnothing$. This individual has a standard exponential lifetime $\pi_\varnothing$, and at the end of its lifetime dies and is replaced by $\ell_\varnothing$ children. These children are labelled by the one-letter words $(1),\ldots,(\ell_\varnothing)$. More generally, if an individual $w=(w_1,\ldots,w_k)$ has $\ell_w$ children upon death, then these children are labelled by the concatenated words $
w1,\ldots,w\ell_w.$
Write $v<w$ if $v$ is a strict ancestor of $w$. Thus, if $w=(w_1,\ldots,w_k)$, then $
v=(w_1,\ldots,w_j)$
for some $0\leq j<k$, with the understanding that $j=0$ refers to the empty word. Write $v\leq w$ if either $v<w$ or $v=w$.

Only labels admitted recursively by the offspring numbers correspond to individuals in the realised tree. For all other labels we set $b_w=d_w=\infty$. For an individual $w$ in the realised tree, its birth and death times are given by
\begin{align*}
b_w
:=
\sum_{v<w}\pi_v
\qquad\text{and}\qquad
d_w
:=
\sum_{v\leq w}\pi_v,
\end{align*}
respectively. The set of individuals alive at time $t$ is
\begin{align*}
\mathcal{N}_t
:=
\{w\in\mathbb{T}:b_w\leq t<d_w\},
\end{align*}
and the number of individuals alive at time $t$ is
\begin{align*}
N_t
:=
\#\mathcal{N}_t.
\end{align*}
The Ulam--Harris labelling makes sense regardless of whether $N_t$ is finite or infinite.

More generally, given an individual $w$ born at time $b_w$, for $t\geq b_w$ let
\begin{align*}
N_t^w
&:=
\text{number of time-$t$ descendants of $w$}
\\
&=
\#\{v\in\mathbb{T}:v\geq w,\ b_v\leq t<d_v\}.
\end{align*}
We define the explosion time of the descendant tree rooted at $w$ by
\begin{align*}
\tau_w
:=
\inf\{t\geq b_w:N_t^w=\infty\}.
\end{align*}
In particular,
\begin{align*}
\tau_\varnothing=\tau
\end{align*}
is the explosion time of the entire process.

We let $(\mathcal{F}_t)_{t\geq0}$ denote the natural filtration generated by all births and deaths up to time $t$, together with the offspring numbers revealed at death. More precisely, setting the birth and death times of labels which do not occur in the tree equal to infinity, we may write
\begin{align*}
\mathcal{F}_t
:=
\sigma\left(
b_w\wedge t,\,
d_w\wedge t,\,
\ell_w\mathbf{1}_{\{d_w\leq t\}}
:
w\in\mathbb{T}
\right).
\end{align*}

Suppose that $t<\tau<\infty$. Since $\mathcal{N}_t$ is then finite, we may define
\begin{align} \label{eq:defspine}
\xi_t
:=
\operatorname*{arg\,min}_{w\in\mathcal{N}_t}\tau_w
\end{align}
to be the individual alive at time $t$ that is destined to have infinitely many descendants before any other individual at time $t$. More specifically, conditional on $\mathcal{F}_t$, the future descendant trees rooted at the individuals in $\mathcal{N}_t$ are independent. Moreover, by the memoryless property of the exponential distribution, for each $w\in\mathcal{N}_t$ the random variable $\tau_w-t$ has the same law as the explosion time of a process started from one individual. Its distribution restricted to finite times has the continuous density $E'$. It follows that, on the event $\{\tau<\infty\}$, the minimum in \eqref{eq:defspine} is attained by a unique individual almost surely. We may break ties lexicographically on the null event that a tie occurs.

The choices $(\xi_t)_{0\leq t<\tau}$ are genealogically consistent. Indeed, if
\begin{align*}
0\leq s<t<\tau,
\end{align*}
then
\begin{align*}
\xi_s\leq\xi_t.
\end{align*}
To see this, observe that every individual alive at time $t$ is descended from a unique individual alive at time $s$. Since the descendant tree rooted at $\xi_s$ is the first of the time-$s$ descendant trees to explode, the first time-$t$ descendant tree to explode must be contained within it. Thus $(\xi_t)_{0\leq t<\tau}$ traces a single infinite ancestral line, which we call the \emph{spine to explosion}.

We define the enriched filtration by
\begin{align*}
\mathcal{F}_t^1
:=
\sigma\left(
\mathcal{F}_t,\,
(\xi_s)_{0\leq s\leq t}
\right),
\end{align*}
on the event $\{t<\tau<\infty\}$. This filtration observes the Galton--Watson tree up to time $t$, together with the identity of the individual lying on the first explosive lineage.

We are particularly interested in conditioning on the event $\{\tau=T\}$. As before, this is understood through a regular conditional distribution, which we denote by
\begin{align*}
\mathbf{P}^T(\,\cdot\,)
:=
\mathbf{P}(\,\cdot\mid\tau=T).
\end{align*}
Under $\mathbf{P}^T$, the process admits a description as a time-inhomogeneous two-type branching process. The individual $\xi_t$ is said to have type $1$, or to be a spine individual, while every individual in $\mathcal{N}_t\setminus\{\xi_t\}$ has type $0$, or is a nonspine individual.

For $\varepsilon>0$, write
\begin{align*}
v(\varepsilon)
:=
1-E(\varepsilon).
\end{align*}
The offspring law of a nonspine individual with remaining time $\varepsilon$ before explosion is the exponentially tilted law
\begin{align} \label{eq:nonspineoffspring}
p_k^{(0)}(\varepsilon)
:=
\frac{p_kv(\varepsilon)^k}{f(v(\varepsilon))},
\qquad k\geq0.
\end{align}
Thus the original offspring law is tilted by the factor $v(\varepsilon)^k$, favouring smaller families. Its generating function is
\begin{align} \label{eq:nonspinepgf}
h_\varepsilon(s)
:=
\sum_{k\geq0}p_k^{(0)}(\varepsilon)s^k
=
\frac{f(v(\varepsilon)s)}{f(v(\varepsilon))}.
\end{align}

The corresponding offspring law of a spine individual is
\begin{align} \label{eq:spineoffspring}
p_k^{(1)}(\varepsilon)
:=
\frac{kp_kv(\varepsilon)^{k-1}}{f'(v(\varepsilon))},
\qquad k\geq1.
\end{align}
This is both size-biased by the factor $k$ and exponentially tilted by the factor $v(\varepsilon)^{k-1}$. Its generating function is
\begin{align} \label{eq:spinepgf}
g_\varepsilon(s)
:=
\sum_{k\geq1}p_k^{(1)}(\varepsilon)s^k
=
\frac{s f'(v(\varepsilon)s)}{f'(v(\varepsilon))}.
\end{align}
The choice of offspring distributions for nonspine and spine individuals is formalised in the following result.
\begin{proposition} \label{prop:spinedecomposition}
Fix $T>0$ and let $0\leq t<T$. Set
\begin{align*}
\varepsilon:=T-t.
\end{align*}
Conditional on $\mathcal{F}_t^1$ under $\mathbf{P}^T$, the future evolution of the process is described as follows.
\begin{enumerate}
\item Each nonspine individual branches at rate
\begin{align} \label{eq:nonspinerate}
r_0(\varepsilon)
:=
\frac{f(v(\varepsilon))}{v(\varepsilon)}.
\end{align}
Conditional on branching, its offspring number has law $(p_k^{(0)}(\varepsilon))_{k\geq0}$ given by \eqref{eq:nonspineoffspring}. All of its children are nonspine individuals.

\item The unique spine individual branches at rate
\begin{align} \label{eq:spinerate}
r_1(\varepsilon)
:=
f'(v(\varepsilon)).
\end{align}
Conditional on branching, its offspring number has law $(p_k^{(1)}(\varepsilon))_{k\geq1}$ given by \eqref{eq:spineoffspring}. Conditional on there being $k$ children, one of the $k$ children is chosen uniformly to continue the spine. The remaining $k-1$ children are nonspine individuals.

\item Conditional on the branching times, offspring numbers and identity of the new spine child, the descendant trees rooted at the resulting children are independent. The new spine descendant tree is conditioned to explode at time $T$, while every nonspine descendant tree is conditioned not to explode before time $T$.
\end{enumerate}
In particular, under $\mathbf{P}^T$ the population is a time-inhomogeneous two-type branching process, whose rates and offspring laws at time $t$ depend on $t$ only through the remaining time $\varepsilon=T-t$.
\end{proposition}

\begin{proof}
Fix $0\leq t<T$ and write $\varepsilon=T-t$. In the unlabelled process, with knowledge of $\mathcal{F}_t$ but not $\mathcal{F}^1_t$, a particle gives birth at unit rate and with offspring probabilities $(p_k)_{k \geq 0}$. The probability this particle does not have infinitely many  descendants at a time $\varepsilon$ in the future is $v(\varepsilon) = 1 - E(\varepsilon)$. In particular,  it follows that, under the conditional law, if $w$ is a nonspine particle at time $t$, then  the Markov property and the branching property give
\begin{align*}
&\mathbf{P}^T\left(
\text{$w$ branches during $(t,t+\delta]$
and has $k$ children} |
\mathcal{F}_t^1
\right)\\
&= \frac{ \mathbf{P}\left(
\text{$\varnothing$ branches during $(0,0+\delta]$
and has $k$ children} , N_\varepsilon < \infty  \right)}{ \mathbf{P}( N_\varepsilon < \infty)} \\
&=
\delta\frac{p_kv(\varepsilon)^k}{v(\varepsilon)}
+o(\delta).
\end{align*}
Summing over $k\geq0$ shows that the total branching rate of a nonspine individual is
\begin{align*}
\frac{1}{v(\varepsilon)}
\sum_{k\geq0}p_kv(\varepsilon)^k
=
\frac{f(v(\varepsilon))}{v(\varepsilon)}.
\end{align*}
Conditional on a branching event, its offspring law is therefore $p_kv(\varepsilon)^k/f(v(\varepsilon))$,
which proves the first part of the proposition.

We next consider the spine individual $\xi_t$. Its descendant tree is conditioned to explode at the precise time $T$. The density of this event, measured from time $t$, is $E'(\varepsilon)$. More generally, given that there are $k$ individuals alive at time $t$, the density of the event that explosion happens precisely at  time $T$ is given by $kE'(\varepsilon)v(\varepsilon)^{k-1}$. It follows by a similar argument that 
\[
\begin{split}
\mathbf{P}^T\left(
\text{$\xi_t$ branches during $(t,t+\delta]$
and has $k$ children} |~
\mathcal{F}_t^1
\right) &= 
\delta\frac{ kp_kE'(\varepsilon) v(\varepsilon)^{k-1}}{E'(\varepsilon)}
+o(\delta)\\
&=
\delta kp_kv(\varepsilon)^{k-1}+o(\delta).
\end{split}
\]
Summing over $k\geq1$ shows that the spine branches at total rate
\begin{align*}
\sum_{k\geq1}kp_kv(\varepsilon)^{k-1}
=
f'(v(\varepsilon)).
\end{align*}
Conditional on a spine branching event, the offspring law is therefore $kp_kv(\varepsilon)^{k-1}/f'(v(\varepsilon))$,
thereby completing the proof.
\end{proof}

\section{Proof of Theorem \ref{thm:spinethm}: The spine as we approach explosion} \label{sec:A}

According to Proposition \ref{prop:spinedecomposition}, under $\mathbf{P}^T$, at time $t=T-\varepsilon$ and conditional on $\mathcal{F}^1_t$, the spine particle gives birth at rate $f'(v(\varepsilon))$, and the offspring generating function at such a birth is
\begin{align*}
g_\varepsilon(s)
=
\frac{s f'(v(\varepsilon)s)}{f'(v(\varepsilon))}.
\end{align*}
We now study the small-$\varepsilon$ asymptotics of these quantities, with a view to proving Theorem \ref{thm:spinethm}.

\begin{proof}[Proof of Theorem \ref{thm:spinethm}]
First, we claim that
\begin{align} \label{eq:spinerateasymptotic}
\lim_{\varepsilon\downarrow0}
\varepsilon f'(v(\varepsilon))
=
\frac{\alpha}{1-\alpha}.
\end{align}
To see this, observe that $1-f(1-x)
=
x+x^\alpha L(x)
\sim
x^\alpha L(x)$ 
as $x\downarrow0$. Since $f'$ is non-decreasing, the monotone density theorem for regularly varying functions gives
\begin{align} \label{eq:fprimeRV}
f'(1-x)
\sim
\alpha x^{\alpha-1}L(x)
\end{align}
as $x\downarrow0$. On the other hand,  \eqref{eq:EODE} states that $E'(\varepsilon)
=
E(\varepsilon)^\alpha L(E(\varepsilon)),
$ and \eqref{eq:Ederivativeratio} states that $\lim_{\varepsilon\downarrow0}
\frac{\varepsilon E'(\varepsilon)}{E(\varepsilon)}
=
\frac{1}{1-\alpha}.
$
Since $1-v(\varepsilon)=E(\varepsilon)$, it follows from
\eqref{eq:fprimeRV} and \eqref{eq:Ederivativeratio} that
\begin{align*}
\varepsilon f'(v(\varepsilon))
\sim
\alpha\varepsilon
E(\varepsilon)^{\alpha-1}L(E(\varepsilon))
=
\alpha\frac{\varepsilon E'(\varepsilon)}{E(\varepsilon)}
\longrightarrow
\frac{\alpha}{1-\alpha},
\end{align*}
which proves \eqref{eq:spinerateasymptotic}.

More generally, by the uniform convergence theorem for regularly varying functions, for $q$ in compact subsets of $(0,\infty)$ we have the uniform convergence 
\begin{align} \label{eq:spinerateq}
\varepsilon f'(v(q\varepsilon))
\longrightarrow
\frac{\alpha}{1-\alpha}\frac{1}{q}.
\end{align}
Since the spinal event rate is deterministic and depends only on the remaining time, the spinal event times form an inhomogeneous Poisson process. Integrating \eqref{eq:spinerateq} over $q\in[s,t]$ proves the first part of Theorem \ref{thm:spinethm}.

We turn to studying the size of the spine births as we approach explosion. Let $K_\varepsilon$ denote the size of a spine birth at a time $\varepsilon$ before explosion. By the definition of the spine offspring generating function,
\begin{align}
\mathbf{E}\left[
e^{-\theta E(\varepsilon)K_\varepsilon}
\right]
=
e^{-\theta E(\varepsilon)}
\frac{
f'\left(v(\varepsilon)e^{-\theta E(\varepsilon)}\right)
}{
f'(v(\varepsilon))
}.
\label{eq:spinemarkLT}
\end{align}
Now, 
\begin{align} \label{eq:jar}
1-v(\varepsilon)e^{-\theta E(\varepsilon)}
= 
1-(1-E(\varepsilon))e^{-\theta E(\varepsilon)}= 
(1+\theta)E(\varepsilon)
+O(E(\varepsilon)^2).
\end{align}
By \eqref{eq:fprimeRV} and \eqref{eq:jar} we have 
\[
\begin{split}
\lim_{\varepsilon \downarrow 0} \mathbf{E}\left[
e^{-\theta E(\varepsilon)K_\varepsilon}
\right]
&=\lim_{\varepsilon \downarrow 0} 
e^{-\theta E(\varepsilon)}
\frac{
f'\left(v(\varepsilon)e^{-\theta E(\varepsilon)}\right)
}{
f'(v(\varepsilon))
} =(1+\theta)^{\alpha-1},
\end{split}
\]
completing the proof.

\end{proof}

\section{Proof of Theorem \ref{thm:main}: The population size as we approach explosion} \label{sec:B}
In this section we prove Theorem \ref{thm:main} and undertake further calculations which will help shed light on the result. Recall that conditioned on the explosion of the continuous-time Galton--Watson tree at time $\tau = T$, we define the stochastic process $Z^{\varepsilon,T} := (Z^{\varepsilon,T}_t)_{t \in \mathbb{R}}$ by setting
\begin{align*}
Z^{\varepsilon,T}_t := E(\varepsilon e^{-t}) N_{T-\varepsilon e^{-t}}, \qquad t \geq - \log(T/\varepsilon),
\end{align*}
and setting $Z^{\varepsilon,T}_t = E(T)$ for all $t < - \log(T/\varepsilon)$. 

\subsection{The limiting transition probabilities}
Our next result calculates the limiting transition Laplace transform of this process.

\begin{proposition} \label{prop:transitionlaplace}
Regardless of the value of $T > 0$, as $\varepsilon \downarrow 0$, $Z^{\varepsilon,T}$  converges in finite-dimensional distributions to a Markov process $(Z_t)_{t \in \mathbb{R}}$ whose transition semigroup is characterised by 
\begin{align} \label{eq:transitionlaplace0}
\mathbf{E}\left[
\exp\left\{
-\theta Z_t
\right\}
\,\middle|\,
Z_s = x
\right]=
\exp\left\{
-xu_{t-s}(\theta)
\right\}
\left(
\frac{1+u_{t-s}(\theta)}
{1+\theta}
\right)^\alpha,
\end{align}
where
\begin{align} \label{eq:defutheta}
u_r(\theta)
:=
\left(
1-e^{-r}
+
e^{-r}(1+\theta)^{1-\alpha}
\right)^{\frac{1}{1-\alpha}}-1,
\qquad r\geq0.
\end{align}
\end{proposition}

\begin{proof}

By Proposition \ref{prop:size}, the one-dimensional marginal at any fixed logarithmic time converges to the Gamma$(\alpha,1)$ law. Since, for each $\varepsilon>0$, $Z^{\varepsilon,T}$ is a time-inhomogeneous Markov process, it is therefore sufficient for finite-dimensional convergence to establish convergence of the transition Laplace transforms. Namely, it is sufficient to show that for all $s<t$ in $\mathbb R$, whenever $(n_\varepsilon)_{\varepsilon>0}$ is a sequence of positive integers such that $x_\varepsilon:=E(\varepsilon e^{-s})n_\varepsilon\to x$, we have

\begin{align} \label{eq:transitionlaplace}
&\lim_{\varepsilon\downarrow0}
\mathbf{E}^T\left[
\exp\left\{
-\theta Z^{\varepsilon,T}_t
\right\}
\,\middle|\,
Z^{\varepsilon,T}_s = x_\varepsilon
\right]=
\exp\left\{
-xu_{t-s}(\theta)
\right\}
\left(
\frac{1+u_{t-s}(\theta)}
{1+\theta}
\right)^\alpha,
\end{align}

Note that $Z_t^{\varepsilon e^{-t_0},T} = Z^{\varepsilon, T}_{t +  t_0}$. It follows that to prove the result, we may further assume without loss of generality that $s = 0$ and $0 < t$. 

Now replacing $t = T-\varepsilon$ and $t'=T-\varepsilon e^{-t}$ in \eqref{eq:conditionaltransition} we obtain 
\begin{align}
&\mathbf{E}^T\left[
s^{N_{T-\varepsilon e^{-t}}} \,\middle|\, N_{T-\varepsilon}=n\right]=
\frac{sE'(\varepsilon e^{-t})}{E'(\varepsilon)}
F_{\varepsilon(1-e^{-t})}'(\tilde{s})
\left(
\frac{
F_{\varepsilon(1-e^{-t})}(\tilde{s})
}{
1-E(\varepsilon)
}
\right)^{n-1}, 
\end{align}
where we are using the shorthand $\tilde{s} = s (1 - E(\varepsilon e^{-t}) )$. 
By setting $s_\varepsilon = e^{ - \theta E(\varepsilon e^{-t}) }$ we obtain
\begin{align} \label{eq:nz0}
\mathbf{E}^T\left[
\exp\left\{
-\theta Z^{\varepsilon,T}_t
\right\}
\,\middle|\,
Z^{\varepsilon,T}_{0} = x_\varepsilon
\right]
= \frac{s_\varepsilon E'(\varepsilon e^{-t})}{E'(\varepsilon)}
F_{\varepsilon(1-e^{-t})}'(\tilde{s}_\varepsilon)
\left(
\frac{
F_{\varepsilon(1-e^{-t})}(\tilde{s}_\varepsilon)
}{
1-E(\varepsilon)
}
\right)^{n-1}.
\end{align}
In this case,
\begin{align*}
\tilde{s}_\varepsilon = 1 - (\theta+1) E(\varepsilon e^{-t}) + O(E(\varepsilon)^2).
\end{align*}
As in \eqref{eq:defv}, we choose $v_{t,\theta,\varepsilon}>0$ such that
\begin{align} \label{eq:nicer}
\tilde{s}_\varepsilon = 1 - E(v_{t,\theta,\varepsilon}\varepsilon).
\end{align}
It then follows that
\begin{align*}
\frac{E(v_{t,\theta,\varepsilon}\varepsilon)}
{E(\varepsilon)}
= \frac{1-\tilde{s}_\varepsilon
}{E(\varepsilon)}
=
(\theta+1) \frac{E(\varepsilon e^{-t})}{E(\varepsilon)} +O(E(\varepsilon)).
    \end{align*}
Thus, proceeding as in \eqref{eq:vlimit}, we have 
\begin{align*}
v_{t,\theta,\varepsilon} \to v_{t,\theta} := e^{-t}(\theta+1)^{1-\alpha} \qquad \text{as $\varepsilon \downarrow 0$}.
\end{align*}
Now using \eqref{eq:Esemigroup} and \eqref{eq:nicer} we have
\begin{align} \label{eq:nz1}
F_{\varepsilon(1-e^{-t})}(\tilde{s}_\varepsilon) = 1 - E(\varepsilon (1 - e^{-t} + v_{t,\theta,\varepsilon} ))
\end{align}
and by \eqref{eq:Fderivative} and \eqref{eq:nicer} we have
\begin{align} \label{eq:nz2}
F'_{\varepsilon(1-e^{-t})}(\tilde{s}_\varepsilon) = \frac{E'(\varepsilon (1 - e^{-t} + v_{t,\theta,\varepsilon} ))}{E'(\varepsilon v_{t,\theta,\varepsilon} )}.
\end{align}
Using \eqref{eq:nz1} and \eqref{eq:nz2} in \eqref{eq:nz0} we obtain
\begin{align} \label{eq:nz3}
&\mathbf{E}^T\left[
\exp\left\{
-\theta Z^{\varepsilon,T}_t
\right\}
\,\middle|\,
Z^{\varepsilon,T}_{0} = x_\varepsilon
\right] \nonumber \\
&= \frac{s_\varepsilon E'(\varepsilon e^{-t})}{E'(\varepsilon)}
\frac{E'(\varepsilon (1 - e^{-t} + v_{t,\theta,\varepsilon} ))}{E'(\varepsilon v_{t,\theta,\varepsilon} )}
\left(
\frac{1 - E(\varepsilon (1 - e^{-t} + v_{t,\theta,\varepsilon} ))}{1-E(\varepsilon)}
\right)^{x_\varepsilon/E(\varepsilon)-1}.
\end{align}
Note $s_\varepsilon \to 1$ as $\varepsilon \to 0$. Moreover, \eqref{eq:E'form} states that $E'(t) = t^{\alpha/(1-\alpha)}L_2(t)$ where $L_2$ is slowly varying at zero. It follows that the small-$\varepsilon$ limit of the prefactor is given by  
\begin{align} \label{eq:nz4}
\lim_{\varepsilon \downarrow 0}
\frac{s_\varepsilon E'(\varepsilon e^{-t})}{E'(\varepsilon)}
\frac{E'(\varepsilon (1 - e^{-t} + v_{t,\theta,\varepsilon} ))}{E'(\varepsilon v_{t,\theta,\varepsilon} )}  = (e^{-t}(1-e^{-t}+v_{t,\theta})/v_{t,\theta})^{\frac{\alpha}{1-\alpha}} = \left( \frac{1 + u_t(\theta)}{ 1 + \theta} \right)^\alpha.
\end{align}
Meanwhile, \eqref{eq:Eform} states that $E(t) = t^{\frac{1}{1-\alpha}}L_1(t)$ where $L_1$ is slowly varying at zero. It follows that the small-$\varepsilon$ limit of the latter term is given by 
\begin{align} \label{eq:nz5}
\lim_{\varepsilon \downarrow 0}
\left(
\frac{1 - E(\varepsilon (1 - e^{-t} + v_{t,\theta,\varepsilon} ))}{1-E(\varepsilon)}
\right)^{x_\varepsilon/E(\varepsilon)-1} = e^{ - x\left(-1+(1-e^{-t}+v_{t,\theta})^{1/(1-\alpha)}\right)} = e^{ - xu_t(\theta)}.
\end{align}
Combining \eqref{eq:nz4} and \eqref{eq:nz5} in \eqref{eq:nz3} we obtain the result.
\end{proof}

\subsection{Continuous-state branching processes with immigration} \label{sec:CBI}
In order to interpret the previous result, we now introduce the notion of continuous-state branching processes with immigration, following Kyprianou \cite{kyprianou}. A continuous-state branching process (without immigration) is a Markov process $(Z_t)_{t \geq 0}$ taking values in $[0,\infty) \cup \{\infty\}$, with $\{\infty\}$ absorbing, and satisfying the branching property. That is, if $P_x$ is the law of $(Z_t)_{t \geq 0}$ starting from $Z_0 = x$, then
\begin{align*}
(Z_t)_{t \geq 0} \sim P_x,\quad
(Z'_t)_{t \geq 0} \sim P_{x'} \text{ independently}
\quad\Longrightarrow\quad
(Z_t+Z'_t)_{t \geq 0} \sim P_{x+x'}.
\end{align*}
The branching property entails that
if $E_x$ denotes the expectation operator associated with $P_x$, there exists a function $u_t(\theta)$ such that for all $\theta \geq 0$ we have
\begin{align*}
E_x[e^{-\theta Z_t}] = e^{ - xu_t(\theta)}.
\end{align*}
The function $u_t(\theta)$, which we call the Laplace mechanism, satisfies the semigroup property $u_s(u_t(\theta)) = u_{s+t}(\theta)$. We also note in particular that plainly $u_0(\theta) = \theta$ is the identity.
The central result in the theory of continuous-state branching processes is a L\'evy-Khintchine representation for the Laplace mechanism, stating that $u_t(\theta)$ satisfies the ordinary differential equation
\begin{align*}
\frac{\partial}{\partial t}u_t(\theta) + \psi(u_t(\theta)) = 0,
\end{align*}
where $\psi(\cdot)$ is the Laplace exponent of a spectrally positive L\'evy process, meaning $\psi$ admits a L\'evy-Khintchine representation
\begin{align*}
\psi(\lambda) = - q + a \lambda + \frac{1}{2}\sigma^2 \lambda^2 + \int_{(0,\infty)} (e^{- \lambda u} - 1 + \lambda u \mathrm{1}_{\{ u<1 \} }) \Pi(\mathrm{d}u).
\end{align*}
Here $q\geq0$, $a\in\mathbb R$, $\sigma\geq0$, and $\int_{(0,\infty)}(1\wedge u^2)\Pi(\mathrm du)<\infty$.
The function $\psi$ is also referred to as the branching mechanism.

If $\int_{(0,1)}u\Pi(\mathrm du)<\infty$, the entire compensation term $\lambda u\mathbf1_{\{u<1\}}$ can be removed by adjusting the linear coefficient $a$. This is the finite-variation case used below. Replacing the indicator by $1$, while retaining compensation, instead requires $\int_{[1,\infty)}u\Pi(\mathrm du)<\infty$.
Here $q$ corresponds to a jump to infinity at rate $qx$ when the population size is $x$; in the absence of the other terms, the linear term gives exponential drift $-ax$. The process fluctuates with a Brownian component whose variance is proportional to its size and to $\sigma^2$, and most importantly, when the population size is $x$, at rate $x \Pi(\mathrm{d}u)$ the process experiences positive jumps of size $u$, though these jumps might be compensated by an additional drift factor.
We now discuss continuous-state branching processes with immigration, following Li \cite{li}. In the case where a continuous-state branching process has immigration, there may be a continuous increase in population according to a rate $\delta \geq 0$, and a second measure $\widetilde{\Pi}$ (with the property that $\int_{(0,\infty)} (1\wedge u)\widetilde{\Pi}(\mathrm{d}u)$ is finite) such that the population increases by size $u$ at rate $\widetilde{\Pi}(\mathrm{d}u)$ independently of its current population size. In this case we have
\begin{align*}
E_x[ e^{ - \theta Z_t} ] = \exp \left\{ - x u_t(\theta) - \int_0^t \varphi(u_s(\theta)) \mathrm{d}s \right\},
\end{align*}
where $\varphi$ is given by
\begin{align*}
\varphi(\lambda)
=
\delta\lambda
+
\int_{(0,\infty)}
(1-e^{-\lambda u})
\widetilde{\Pi}(\mathrm{d}u).
\end{align*}
The function $\varphi$ is also referred to as the immigration mechanism.
\subsection{The limit process} \label{sec:limit}
With the discussion in the previous section in mind, we now seek to identify the L\'evy measures for both the branching and immigration mechanisms. In essence, we will now prove Theorem \ref{thm:main}.

\begin{proof}[Proof of Theorem \ref{thm:main}]
We first look to calculate the generator. 
Setting $s = 0$ in \eqref{eq:transitionlaplace0} and differentiating with respect to $t$, we obtain
\begin{equation}
\begin{split}\label{eq:dtran}
\frac{\partial}{\partial t} &\mathbf{E}\left[
\exp\left\{
-\theta Z_t
\right\}
\,\middle|\,
Z_0 = x
\right]\\
&\hspace{2cm}=
\left\{
- x \left( \frac{1 + u_t(\theta) }{1 + \theta} \right)^\alpha
+ \frac{\alpha }{1+\theta}
\left( \frac{1 + u_t(\theta) }{1 + \theta} \right)^{\alpha - 1}
\right\}
e^{ - x u_t(\theta) }
\frac{\partial}{\partial t}u_t(\theta).
\end{split}
\end{equation}
Setting $t = 0$ and using $u_0(\theta) = \theta$, we obtain
\begin{align} \label{eq:dtran2}
\left.
\frac{\partial}{\partial t} \mathbf{E}\left[
\exp\left\{
-\theta Z_t
\right\}
\,\middle|\,
Z_0 = x
\right]
\right|_{t = 0}
=
\left(
- x + \frac{\alpha }{1+\theta}
\right)
e^{ - x\theta}
\left.
\frac{\partial}{\partial t}u_t(\theta)
\right|_{t=0}.
\end{align}
Calculating $\left.\frac{\partial}{\partial t}u_t(\theta)\right|_{t=0}$ using its definition in \eqref{eq:defutheta}, we have
\begin{align} \label{eq:dtran3}
\left.
\frac{\partial}{\partial t}u_t(\theta)
\right|_{t=0}
=
\frac{(1+\theta)^\alpha - (1+\theta)}{1-\alpha}.
\end{align}
Using \eqref{eq:dtran3} in \eqref{eq:dtran2}, and writing $J_t := Z_t - Z_0$, we obtain
\begin{align} \label{eq:dtran4}
\left.
\frac{\partial}{\partial t} \mathbf{E}\left[
\exp\left\{
-\theta J_t
\right\}
\,\middle|\,
Z_0 = x
\right]
\right|_{t = 0}
&=
\left(
- x + \frac{\alpha }{1+\theta}
\right)
\frac{(1+\theta)^\alpha-(1+\theta)}{1-\alpha}
\nonumber\\
&=
x\frac{(1+\theta)-(1+\theta)^\alpha}{1-\alpha}
-
\frac{\alpha}{1-\alpha}
\left(
1-(1+\theta)^{\alpha-1}
\right).
\end{align}
Thus, defining
\begin{align}
\psi(\theta)
&:=
\frac{(1+\theta)-(1+\theta)^\alpha}{1-\alpha},
&
\varphi(\theta)
&:=
\frac{\alpha}{1-\alpha}
\left(
1-(1+\theta)^{\alpha-1}
\right),
\end{align}
we have
\begin{align}
\left.
\frac{\partial}{\partial t} \mathbf{E}\left[
e^{-\theta J_t}
\,\middle|\,
Z_0 = x
\right]
\right|_{t=0}
=
x\psi(\theta)-\varphi(\theta).
\end{align}
Equivalently, if $\mathcal{L}$ denotes the generator of $(Z_t)_{t\geq0}$, then
\begin{align}
\mathcal{L}e^{-\theta x}
=
\left(
x\psi(\theta)-\varphi(\theta)
\right)e^{-\theta x}.
\end{align}
This identifies $(Z_t)_{t\geq0}$ as a continuous-state branching process with branching mechanism $\psi$ and immigration mechanism $\varphi$.

We now find L\'evy--Khintchine representations for the two exponents $\psi(\theta)$ and $\varphi(\theta)$. Since $\alpha\in(0,1)$, the L\'evy measure associated with the  branching mechanism $\psi$ satisfies
\begin{align*}
\int_{(0,1)} u\Pi(\mathrm{d}u)<\infty,
\end{align*}
and hence we may use the form of the L\'evy--Khintchine representation in which the truncation term is absorbed into the linear drift. For the immigration mechanism, we recognise straightforwardly that
\begin{align*}
\varphi(\theta)
=
\frac{\alpha}{1-\alpha}
\left(
1-(1+\theta)^{-(1-\alpha)}
\right)
=
\frac{\alpha}{1-\alpha}
\int_0^\infty
(1-e^{-\theta u})
\gamma_{1-\alpha}(u)\,\mathrm{d}u,
\end{align*}
where
\begin{align*}
\gamma_\kappa(u)
:=
\frac{u^{\kappa-1}e^{-u}}{\Gamma(\kappa)},
\qquad u>0,
\end{align*}
is the density of a $\Gamma(\kappa,1)$ random variable. Thus the L\'evy measure associated with the  immigration mechanism $\varphi$ is
\begin{align} \label{eq:immigrationlevy}
\widetilde{\Pi}(\mathrm{d}u)
=
C_\alpha u^{-\alpha}e^{-u}\,\mathrm{d}u,
\qquad
C_\alpha
:=
\frac{\alpha}{\Gamma(2-\alpha)}.
\end{align}
In particular,
\begin{align*}
\widetilde{\Pi}((0,\infty))
=
\frac{\alpha}{1-\alpha},
\end{align*}
so the immigration component is a compound Poisson process whose jumps occur at rate $\alpha/(1-\alpha)$ and whose jump sizes have the $\Gamma(1-\alpha,1)$ distribution.

For the branching mechanism, first note that
\begin{align*}
\int_0^\infty
(1-e^{-\theta u})u^{-1-\alpha}e^{-u}\,\mathrm{d}u
=
\int_0^\theta
\int_0^\infty
u^{-\alpha}e^{-(1+r)u}\,\mathrm{d}u\,\mathrm{d}r =
\frac{\Gamma(1-\alpha)}{\alpha}
\left(
(1+\theta)^\alpha-1
\right).
\end{align*}
It follows that
\begin{align*}
\psi(\theta)
=
\frac{\theta}{1-\alpha}
+
C_\alpha
\int_0^\infty
(e^{-\theta u}-1)
u^{-1-\alpha}e^{-u}\,\mathrm{d}u.
\end{align*}
Consequently, the L\'evy measure associated with the  branching mechanism $\psi$ is
\begin{align} \label{eq:branchinglevy}
\Pi(\mathrm{d}u)
=
C_\alpha u^{-1-\alpha}e^{-u}\,\mathrm{d}u.
\end{align}
Notice in particular the simple relation
\begin{align*}
\widetilde{\Pi}(\mathrm{d}u)
=
u\Pi(\mathrm{d}u).
\end{align*}

Equivalently, the branching and immigration mechanisms satisfy
\begin{align*}
\varphi(\theta)=\psi'(\theta)-\psi'(0).
\end{align*}
This size-bias relation is the continuous-state counterpart of the distinguished-spine decomposition above.

We may therefore write the generator of the limiting process, for suitable test functions $h$, as
\begin{align} \label{eq:limitgenerator}
\mathcal{L}h(x)
&=
-\frac{x}{1-\alpha}h'(x)
+
C_\alpha x
\int_0^\infty
\left(
h(x+u)-h(x)
\right)
u^{-1-\alpha}e^{-u}\,\mathrm{d}u
\nonumber\\
&\hspace{5cm}\quad+
C_\alpha
\int_0^\infty
\left(
h(x+u)-h(x)
\right)
u^{-\alpha}e^{-u}\,\mathrm{d}u.
\end{align}
Thus the first integral describes jumps generated by the existing population, with branching L\'evy measure $\Pi$, while the second describes jumps arriving through immigration, with immigration L\'evy measure $\widetilde{\Pi}$.
\end{proof}

\subsection{The stationary distribution}

Recall that Proposition \ref{prop:size} states that under $\mathbf{P}^T$, the random variable $E(\varepsilon)N_{T-\varepsilon}$ converges to a $\mathrm{Gamma}(\alpha,1)$ distributed random variable. We now verify stationarity of this limiting law by hand. Let
\begin{align*}
\pi_\alpha(\mathrm{d}x)
=
\frac{x^{\alpha-1}e^{-x}}{\Gamma(\alpha)}\,\mathrm{d}x
\end{align*}
denote the $\mathrm{Gamma}(\alpha,1)$ distribution. To verify that $\pi_\alpha$ is stationary, it is sufficient to check that
\begin{align*}
\int_0^\infty \mathcal{L}h(x)\pi_\alpha(\mathrm{d}x)=0
\end{align*}
for a suitable class of test functions $h$. We verify this for the functions $h(x)=e^{-\theta x}$, $\theta\geq0$, which characterise the distribution. Recall that
\begin{align*}
\mathcal{L}e^{-\theta x}
=
\left(
x\psi(\theta)-\varphi(\theta)
\right)e^{-\theta x},
\end{align*}
where
\[
\psi(\theta)
=
\frac{(1+\theta)-(1+\theta)^\alpha}{1-\alpha}
\quad\textrm{ and }\quad
\varphi(\theta)
=
\frac{\alpha}{1-\alpha}
\left(
1-(1+\theta)^{\alpha-1}
\right).
\]
The Laplace transform of $\pi_\alpha$ is
\begin{align*}
\int_0^\infty e^{-\theta x}\pi_\alpha(\mathrm{d}x)
=
(1+\theta)^{-\alpha},
\end{align*}
and therefore
\begin{align*}
\int_0^\infty xe^{-\theta x}\pi_\alpha(\mathrm{d}x)
=
\alpha(1+\theta)^{-\alpha-1}.
\end{align*}
It follows that
\begin{align*}
\int_0^\infty
\mathcal{L}e^{-\theta x}
\pi_\alpha(\mathrm{d}x)
&=
\alpha\psi(\theta)(1+\theta)^{-\alpha-1}
-
\varphi(\theta)(1+\theta)^{-\alpha}
\\
&=
\frac{\alpha}{1-\alpha}
\left(
(1+\theta)^{-\alpha}-(1+\theta)^{-1}
\right)
\\
&\qquad-
\frac{\alpha}{1-\alpha}
\left(
(1+\theta)^{-\alpha}-(1+\theta)^{-1}
\right)
\\
&=0.
\end{align*}
Thus $\pi_\alpha$ is stationary for $(Z_t)_{t\geq0}$, as claimed.

\subsection{Campbell's formula and the stationary distribution} \label{sec:campbell}

Recall that we observed, following the statement of Theorem \ref{thm:main}, that we have the distributional identity 
\begin{align*}
\sum_{i \geq 1} B_i = Z_0 \sim \mathrm{Gamma}(\alpha,1),
\end{align*} 
where $B_i=Y^i_{T_i}$ are the immigrant-family contributions defined after Theorem \ref{thm:main}. Recall that $(-T_i)$ form a Poisson process of rate $\alpha/(1-\alpha)$, the marks $A_i$ are independent Gamma$(1-\alpha,1)$ variables independent of these times, and, conditional on the times and marks, the $Y^i$ are independent CSBPs with Laplace exponent \eqref{eq:ut} started from $A_i$.

In particular, we can describe the conditional law of $B_i$ given $A_i$ and $T_i$ through its Laplace transform 
\begin{align*}
\mathbf{E}[ e^{ - \theta B_i} | T_i, A_i ] = \mathbf{E}_{A_i}[ e^{ - \theta Y_{T_i} } ] = \exp\{ - A_i u_{T_i}(\theta) \},
\end{align*}
where $u_t(\theta)$ is the underlying Laplace exponent of the subcritical CSBP.

Note that $\mathbf{E}[ e^{ - \lambda A_i} ] = (1+\lambda)^{-(1-\alpha)}$. 
In particular, averaging over $A_i$ we obtain 
\begin{align*}
\mathbf{E}[ e^{ - \theta B_i} | T_i ] = (1 + u_{T_i}(\theta) )^{-(1-\alpha)}.
\end{align*}
In particular, we can write
\begin{align*}
\mathbf{E}[ e^{ - \theta Z_0 } ] = \mathbf{E} \left[ \exp \left\{ \sum_{i \geq 1} g_\theta(T_i) \right\} \right] 
\end{align*}
where $g_\theta(t) = \log \left((1 + u_t(\theta) )^{-(1-\alpha)}\right)$. Now by the Laplace functional of a Poisson process \cite[Theorem 3.9]{LP} we have
\begin{align} \label{eq:intcalc0}
\mathbf{E}[ e^{ - \theta Z_0 } ] = \exp \left\{ \int_0^\infty \lambda(t)( e^{ g_\theta(t)} - 1 ) \mathrm{d}t \right\} = \exp \left\{ \frac{\alpha}{1-\alpha} \int_0^\infty ((1+u_t(\theta))^{-(1-\alpha)}- 1 ) \mathrm{d}t \right\},
\end{align}
where we have used the fact that events a time $t$ before $0$ occurred at constant rate $\lambda(t) = \alpha/(1-\alpha)$. 

Now, using the fact that $u_t(\theta) = \left( 1 + e^{-t}((1+\theta)^{1-\alpha} - 1) \right)^{1/(1-\alpha)} - 1$, we have
\begin{align} \label{eq:intcalc} 
\int_0^\infty ((1+u_t(\theta))^{-(1-\alpha)}- 1 ) \mathrm{d}t &= - \int_0^\infty \frac{ e^{-t}((1+\theta)^{1-\alpha}-1) }{1 + e^{-t}((1+\theta)^{1-\alpha}-1)} \mathrm{d}t \nonumber \\
&= + \int_0^\infty \frac{\mathrm{d}}{\mathrm{d}t} \log( 1 + e^{-t}((1+\theta)^{1-\alpha} - 1 ) ) \mathrm{d}t \nonumber\\
&= - (1-\alpha)\log(1+\theta).
\end{align} 
Using \eqref{eq:intcalc} in \eqref{eq:intcalc0} we obtain 
\begin{align*}
\mathbf{E}[ e^{ - \theta Z_0} ] = (1 + \theta)^{ - \alpha},
\end{align*} 
confirming the fact that $Z_0$ is Gamma distributed with parameter $\alpha$.

\subsection{The CBI as a size-biased transform of the subcritical CSBP}

We close this section by noting briefly that there is an interesting relationship between the CBI $(Z_t)_{t \geq 0}$ and $(Y_t)_{t \geq 0}$: namely, $(Z_t)_{t \geq 0}$ starting from $x > 0$ has the law of $(Y_t)_{t \geq 0}$ starting from $x > 0$ but size-biased at each time. To see this, under $\mathbf{P}_x$ let $(Y_t)_{t \geq 0}$ be a CSBP with Laplace exponent $u_t(\theta)$ (so that $\mathbf{P}_x [ e^{ - \theta Y_t} ]  = e^{ - x u_t(\theta)}$) and define the change of measure $\mathbf{Q}_x$ by 
\begin{align*}
\left.\frac{\mathrm d\mathbf Q_x}{\mathrm d\mathbf P_x}\right|_{\mathcal F_t}
=\frac{e^tY_t}{x},
\end{align*}
where we note that since $\mathbf{P}_x [ Y_t ] = xe^{-t}$, the latter is a unit-mean martingale. 
Now note that we can calculate the Laplace transform of $Y_t$ under $\mathbf{Q}_x$ directly. Indeed, we have
\begin{align*}
\mathbf{Q}_x [ e^{ - \theta Y_t} ] = \frac{e^t}{x} \mathbf{P}_x [ Y_t e^{ - \theta Y_t} ] = - \frac{e^t}{x} \frac{\partial}{\partial \theta} e^{ - x u_t(\theta) } = e^t u_t'(\theta) e^{ - x u_t(\theta)}.
\end{align*}
Now note that $e^t u_t'(\theta) = \left( \frac{1+u_t(\theta)}{1+\theta}\right)^\alpha$, cf.\ \eqref{eq:nz4}. This is the immortal-spine transform associated with the CSBP $Q$-process; see \cite[Section 12.3.2]{kyprianou}. Here $Y$ extinguishes without hitting zero, so the transform should not be interpreted as literal conditioning on positivity at every finite time.


\section{Proofs of Theorems \ref{thm:coal} and \ref{thm:coal2}: Coalescent structure of the limiting process} \label{sec:C}

\subsection{Time reversals of Markov processes}

When $Z = (Z_t)_{t \in \mathbb{R}}$ is a stationary real-valued càdlàg time-homogeneous Markov process with stationary distribution $\pi$, we can consider its time reversal $R=(R_t)_{t\in\mathbb R}$, where
\begin{align*}
R_t := Z_{(-t)-}, \qquad t\in\mathbb R.
\end{align*}
This time-reversed process $R$ is again a stationary càdlàg Markov process with stationary distribution $\pi$, and its generator is the $L^2(\pi)$-adjoint of that of $Z$.

In particular, suppose that $\pi$ has a density, also denoted by $\pi$, and that $Z$ has jump kernel
\begin{align*}
J(x,\mathrm{d}y)=p(x,y)\,\mathrm{d}y,
\end{align*}
meaning that its jump measure has predictable compensator
$p(Z_{t-},y)\,\mathrm{d}y\,\mathrm{d}t$.
Then the time-reversed jump density is given by the balance equation
\begin{align} \label{eq:reversal}
\pi(x) p(x,y) = \pi(y) p^{\mathrm{rev}}(y,x),
\end{align}
for almost every $(x,y)$ off the diagonal; see, e.g. \cite[Equation (5.13)]{CL}. This determines
$p^{\mathrm{rev}}(y,x)$ wherever $\pi(y)>0$.
For further information on time-reversal of Markov processes, the reader might consult \cite{CW}, \cite{CW2} or \cite{nagasawa}.


\subsection{Transition rates for the time-reversed Markov process}

Let us take stock of our setting. In Theorem \ref{thm:main} we introduced a stationary Markov process $(Z_t)_{t \geq 0}$ with only positive jumps and with a $\mathrm{Gamma}(\alpha,1)$ stationary distribution. We will see in this section that the time-reversed process has a pleasant description, and that we can moreover keep track of the contribution of spinal and nonspinal events to its jump rates.

We can read off from \eqref{eq:ngen} that due to the branching component, the forwards-in-time process $(Z_t)_{t \geq 0}$ has exponential decay at rate $\frac{1}{1-\alpha}$ together with jumps from $x$ to $\mathrm{d}y$ (with $x < y$) at rate
\begin{align*}
p_{\mathrm{n}}(x,y)\mathrm{d}y := C_\alpha x (y-x)^{-(1+\alpha)} e^{ - (y-x)}\mathrm{d}y.
\end{align*}
Likewise, due to the spinal immigration component, the forwards-in-time process has jumps from $x$ to $\mathrm{d}y$ at rate 
\begin{align*}
p_{\mathrm{s}}(x,y)\mathrm{d}y := C_\alpha (y-x)^{ - \alpha} e^{ - (y-x)}\mathrm{d}y.
\end{align*}

In our case, in the setting of \eqref{eq:reversal} we can attribute the jumps backwards-in-time to either branching or immigration jumps. Accordingly, we can calculate that, backwards in time, we have jumps from $y$ to $\mathrm{d}x$ (with $y > x$) due to branching events at rate
\begin{align*}
p^{\mathrm{rev}}_n(y,x)\mathrm{d}x := \frac{\pi(x)}{\pi(y)} p_{\mathrm n}(x,y)\,\mathrm{d}x = C_\alpha x^\alpha (y-x)^{-(1+\alpha)} y^{ 1-\alpha}\mathrm{d}x,
\end{align*}
and due to spinal immigration events at rate 
\begin{align*}
p^{\mathrm{rev}}_s(y,x)\mathrm{d}x := \frac{\pi(x)}{\pi(y)} p_{\mathrm s}(x,y)\,\mathrm{d}x = C_\alpha x^{\alpha-1} (y-x)^{-\alpha} y^{ 1-\alpha}\mathrm{d}x.
\end{align*}
The total rate of jumps from $y$ to $\mathrm{d}x$ is $p^{\mathrm{rev}}(y,x)\mathrm{d}x$ where 
\begin{align*}
    p^{\mathrm{rev}}(y,x) := p^{\mathrm{rev}}_n(y,x) + p^{\mathrm{rev}}_s(y,x).
\end{align*}
Of course, backwards in time, there is also continuous exponential \emph{growth} of the population at rate $1/(1-\alpha)$. 

Let us highlight a key structural property. Namely, given that a jump from $y$ to $\mathrm{d}x$ occurs at some point, the probability it was due to spinal immigration is given by
\begin{align} \label{eq:ratio}
\frac{p_s^{\mathrm{rev}}(y,x)}{p_s^{\mathrm{rev}}(y,x) + p_n^{\mathrm{rev}}(y,x) } = \frac{y-x}{y}.
\end{align}
Let $v = (y-x)/y$ be the fraction of population size lost. 
From a position $y$, the rate at which a fraction $\mathrm{d}v$ of the population size is lost is given by 
\begin{align*}
yp^{\mathrm{rev}}(y,(1-v)y)\,\mathrm{d}v = C_\alpha y^{1-\alpha}v^{-1-\alpha}(1-v)^{\alpha-1}\,\mathrm{d}v = \Gamma(1+\alpha)y^{1-\alpha}\frac{\mathcal B_{2-\alpha,\alpha}(\mathrm{d}v)}{v^2}
\end{align*}
where 
\begin{align*}
\mathcal{B}_{2-\alpha,\alpha}(\mathrm{d}v) = \frac{\Gamma(2)}{\Gamma(2-\alpha)\Gamma(\alpha)} v^{1 - \alpha} (1-v)^{ \alpha - 1} \mathrm{d}v \qquad v \in [0,1],
\end{align*}
is the probability measure of a Beta random variable with parameters $2-\alpha$ and $\alpha$.

We are now ready to prove Theorem \ref{thm:coal}.

\begin{proof}[Proof of Theorem \ref{thm:coal}]
Suppose that the coalescent currently has $b$ blocks  and the reverse-time population has size $x$. At a reverse jump which removes a fraction $v$ of the population, each block is marked independently with probability $v$, and all marked blocks merge. Thus, for a particular collection of $j$ blocks, the rate at which precisely those $j$ blocks are marked is
\begin{align*}
\Gamma(1+\alpha)x^{1-\alpha}
\int_0^1
v^{j-2}(1-v)^{b-j}
\mathcal B_{2-\alpha,\alpha}(\mathrm{d}v).
\end{align*}
Consequently, relative to the joint filtration of the reverse population and partition processes, the merger rates are those of a Beta$(2-\alpha,\alpha)$-coalescent multiplied by the common time-dependent factor
$\Gamma(1+\alpha)R_{t-}^{1-\alpha}.$ In particular, define 
\begin{align*}
A(r) :=
\int_0^r
\Gamma(1+\alpha)R_s^{1-\alpha}\,\mathrm{d}s,
\end{align*}
so that $I$ is the right-continuous inverse of $A$.
After the time change $t\mapsto I(t)$, the common factor is removed. Hence, when viewed in its own filtration, $(\pi_{I(t)})_{t\geq0}$ has merger rates
\begin{align*}
\lambda_{b,j}
=
\int_0^1
v^{j-2}(1-v)^{b-j}
\mathcal B_{2-\alpha,\alpha}(\mathrm{d}v),
\end{align*}
and is therefore a standard Beta$(2-\alpha,\alpha)$-coalescent, as required. 
\end{proof}

We now prepare the proof of Theorem \ref{thm:coal2}.
Again, in light of \eqref{eq:ratio} and the fact that $v = (y-x)/y$, given that there is a loss of a fraction $v$ of the population, the probability that this event was due to spinal immigration is $v$, and the probability it was due to branching is $1-v$, \emph{regardless of the population size $y$}. In particular, when the size of the reverse-time population process is $y$, the rate at which a fraction $\mathrm{d}v$ of the population size is lost has the amenable decomposition
\begin{align*}
\Gamma(1+\alpha)y^{1-\alpha}\frac{\mathcal B_{2-\alpha,\alpha}(\mathrm{d}v)}{v^2} = \underbrace{\Gamma(1+\alpha)y^{1-\alpha}v\frac{\mathcal B_{2-\alpha,\alpha}(\mathrm{d}v)}{v^2}}_{\text{immigration}} + \underbrace{\Gamma(1+\alpha)y^{1-\alpha}(1-v)\frac{\mathcal B_{2-\alpha,\alpha}(\mathrm{d}v)}{v^2}}_{\text{branching}}.
\end{align*}

\begin{proof}[Proof of Theorem \ref{thm:coal2}]
By \eqref{eq:ratio} and the fact that $v = (y-x)/y$ is the fraction of the population lost, conditional on a reverse jump of relative size $v$, its forward-time origin is spinal with probability $v$ and nonspinal with probability $1-v$.

After making the time change $t\mapsto I(t)$, the relative jump intensity, with respect to the joint past of the population and coloured partition, is
\begin{align*}
\mathrm{d}t\,v^{-2}\mathcal B_{2-\alpha,\alpha}(\mathrm{d}v).
\end{align*}
We simply multiply this event rate by $v$ to obtain the spinal jumps, and $1-v$ to obtain the nonspinal jumps. Suppose that the coloured partition currently has $b$ blue blocks, and fix a prescribed collection of $j$ of them. The probability that precisely these blue blocks are marked is $v^j(1-v)^{b-j}$.

At a nonspinal event, the marked blue blocks merge into a blue block and the red block is unaffected. Consequently, the rate of this transition is
\begin{align*}
\int_0^1 v^j(1-v)^{b-j}(1-v)v^{-2}
\mathcal B_{2-\alpha,\alpha}(\mathrm{d}v)
=\lambda^{\mathrm{blue}}_{b,j},
\qquad 2\leq j\leq b.
\end{align*}
At a spinal event, the marked blue blocks merge with the red block. The rate of this transition is
\begin{align*}
\int_0^1 v^j(1-v)^{b-j}v\,v^{-2}
\mathcal B_{2-\alpha,\alpha}(\mathrm{d}v)
=\lambda^{\mathrm{red}}_{b,j},
\qquad 1\leq j\leq b.
\end{align*}
These rates are finite and depend only on the current coloured partition. Hence, in its own filtration, $(\widetilde\pi_{I(t)})_{t\geq0}$ is the finite-state continuous-time Markov chain with the transition rates of the red/blue coalescent. This completes the proof, as both processes start from the red singleton $\{0\}$ and the blue singletons $\{1\},\ldots,\{k\}$. 
\end{proof}

\section*{Acknowledgements}  
\noindent 
The authors would like to acknowledge the support of the New Zealand Aotearoa Royal Society Te Ap\={a}rangi Marsden Fund (22-UOA-052).

\section*{AI Disclosure}
ChatGPT-5.6 was used to run the simulation and generate images for Figures \ref{fig:explosive}, \ref{fig:Zt} and \ref{fig:redblue}. AI was also used for proofreading purposes.


\begin{thebibliography}{99}

\bibitem{AbrahamDelmasHe2021}
\textsc{Abraham, R., Delmas, J.-F. and He, H.} (2021).
\newblock Some properties of stationary continuous state branching processes.
\newblock \emph{Stochastic Processes and their Applications} \textbf{141}, 309--343.

\bibitem{amini}
\textsc{Amini, O., Devroye, L., Griffiths, S. and Olver, N.} (2013).
\newblock On explosions in heavy-tailed branching random walks.
\newblock \emph{Annals of Probability} \textbf{41}, 1864--1899.

\bibitem{AN72}
\textsc{Athreya, K. B. and Ney, P. E.} (1972).
\newblock \emph{Branching Processes}.
\newblock Springer-Verlag, Berlin--Heidelberg--New York.

\bibitem{BerestyckiBerestyckiSchweinsberg2007}
\textsc{Berestycki, J., Berestycki, N. and Schweinsberg, J.} (2007).
\newblock Beta-coalescents and continuous stable random trees.
\newblock \emph{Annals of Probability} \textbf{35}, 1835--1887.

\bibitem{BLG}
\textsc{Bertoin, J. and Le Gall, J.-F.} (2000).
\newblock The Bolthausen--Sznitman coalescent and the genealogy of continuous-state branching processes.
\newblock \emph{Probability Theory and Related Fields} \textbf{117}, 249--266.

\bibitem{BGT}
\textsc{Bingham, N. H., Goldie, C. M. and Teugels, J. L.} (1987).
\newblock \emph{Regular Variation}.
\newblock Cambridge University Press, Cambridge.

\bibitem{BirknerEtAl2005}
\textsc{Birkner, M., Blath, J., Capaldo, M., Etheridge, A. M., M\"ohle, M., Schweinsberg, J. and Wakolbinger, A.} (2005).
\newblock Alpha-stable branching and beta-coalescents.
\newblock \emph{Electronic Journal of Probability} \textbf{10}, 303--325.

\bibitem{CW}
\textsc{Chung, K. L. \& Walsh, J. B.} (1969). To reverse a Markov process. \emph{Acta Math}, 123(1), 225-251.

\bibitem{CW2}
\textsc{Chung, K. L. \& Walsh, J. B.} (2005). \emph{Markov processes, Brownian motion, and time symmetry.} New York, NY: Springer New York.

\bibitem{CL}
\textsc{Conforti, G. \& L\'eonard, C.} (2022). Time reversal of Markov processes with jumps under a finite entropy condition. \emph{Stoch. Process. Appl.} \textbf{144}, 85--124.

\bibitem{doney}
\textsc{Doney, R. A.} (1984).
\newblock A note on some results of Schuh.
\newblock \emph{Journal of Applied Probability} \textbf{21}, 192--196.

\bibitem{FoucartHenard2013}
\textsc{Foucart, C. and H\'enard, O.} (2013).
\newblock Stable continuous-state branching processes with immigration and Beta--Fleming--Viot processes with immigration.
\newblock \emph{Electronic Journal of Probability} \textbf{18}, no.~23, 1--21.

\bibitem{FoucartMaMallein2019}
\textsc{Foucart, C., Ma, C. and Mallein, B.} (2019).
\newblock Coalescences in continuous-state branching processes.
\newblock \emph{Electronic Journal of Probability} \textbf{24}, no.~103, 1--52.

\bibitem{grey}
\textsc{Grey, D. R.} (1974).
\newblock Asymptotic behaviour of continuous time, continuous state-space branching processes.
\newblock \emph{Journal of Applied Probability} \textbf{11}, 669--677.

\bibitem{greyagedependent}
\textsc{Grey, D. R.} (1974).
\newblock Explosiveness of age-dependent branching processes.
\newblock \emph{Zeitschrift f\"ur Wahrscheinlichkeitstheorie und Verwandte Gebiete} \textbf{28}, 129--137.

\bibitem{greyexplosion}
\textsc{Grey, D. R.} (1989).
\newblock A note on explosiveness of Markov branching processes.
\newblock \emph{Advances in Applied Probability} \textbf{21}, 226--228.

\bibitem{HarrisJohnstonPardo2024}
\textsc{Harris, S. C., Johnston, S. G. G. and Pardo, J. C.} (2024).
\newblock Universality classes for the coalescent structure of heavy-tailed Galton--Watson trees.
\newblock \emph{Annals of Probability} \textbf{52}, 387--433.

\bibitem{HarrisJohnstonRoberts2020}
\textsc{Harris, S. C., Johnston, S. G. G. and Roberts, M. I.} (2020).
\newblock The coalescent structure of continuous-time Galton--Watson trees.
\newblock \emph{Annals of Applied Probability} \textbf{30}, 1368--1414.

\bibitem{iyerlodewijks}
\textsc{Iyer, T. and Lodewijks, B.} (2023).
\newblock On the structure of genealogical trees associated with explosive Crump--Mode--Jagers branching processes.
\newblock arXiv:2311.14664.

\bibitem{JKR}
\textsc{Johnston, S. G. G., Kyprianou, A. E. and Rogers, T.} (2023).
\newblock Multitype $\Lambda$-coalescents.
\newblock \emph{Annals of Applied Probability} \textbf{33}, 4210--4237.

\bibitem{JL}
\textsc{Johnston, S. G. G. and Lambert, A.} (2023).
\newblock The coalescent structure of uniform and Poisson samples from multitype branching processes.
\newblock \emph{Annals of Applied Probability} \textbf{33}, 4820--4857.

\bibitem{kyprianou}
\textsc{Kyprianou, A. E.} (2014).
\newblock \emph{Fluctuations of L\'evy Processes with Applications: Introductory Lectures}, 2nd ed.
\newblock Springer, Heidelberg.

\bibitem{labbe}
\textsc{Labb\'e, C.} (2013).
\newblock Quasi-stationary distributions associated with explosive CSBP.
\newblock \emph{Electronic Communications in Probability} \textbf{18}, article no.~57, 1--13.

\bibitem{Lambert2002}
\textsc{Lambert, A.} (2002).
\newblock The genealogy of continuous-state branching processes with immigration.
\newblock \emph{Probability Theory and Related Fields} \textbf{122}, 42--70.

\bibitem{LP}
\textsc{Last, G., \& Penrose, M.} (2018). Lectures on the Poisson process (Vol. 7). Cambridge University Press.

\bibitem{li}
\textsc{Li, Z.} (2020).
\newblock Continuous-state branching processes with immigration.
\newblock In Y.~Jiao (ed.), \emph{From Probability to Finance}, Springer, Singapore, 1--69.

\bibitem{LYZ}
\textsc{Li, P.-S., Yang, X. and Zhou, X.} (2019).
A general continuous-state nonlinear branching process.
\emph{Ann. Appl. Probab.} \textbf{29}(4), 2523--2555.

\bibitem{lodewijks}
\textsc{Lodewijks, B.} (2026).
\newblock A star is born: explosive Crump--Mode--Jagers branching processes.
\newblock \emph{Journal of Applied Probability} \textbf{63}(3), 1244--1262.

\bibitem{nagasawa}
\textsc{Nagasawa, M.} (1964).
\newblock Time reversions of Markov processes.
\newblock \emph{Nagoya Mathematical Journal} \textbf{24}, 177--204.

\bibitem{pakes}
\textsc{Pakes, A. G.} (1993).
\newblock Explosive Markov branching processes: entrance laws and limiting behaviour.
\newblock \emph{Advances in Applied Probability} \textbf{25}, 737--756.

\bibitem{sagitov}
\textsc{Sagitov, S. M.} (1996).
\newblock On an explosive branching process.
\newblock \emph{Theory of Probability and Its Applications} \textbf{40}, 575--577.

\bibitem{schuh}
\textsc{Schuh, H.-J.} (1982).
\newblock Sums of i.i.d.\ random variables and an application to the explosion criterion for Markov branching processes.
\newblock \emph{Journal of Applied Probability} \textbf{19}, 29--38.

\bibitem{Schweinsberg2003}
\textsc{Schweinsberg, J.} (2003).
\newblock Coalescent processes obtained from supercritical Galton--Watson processes.
\newblock \emph{Stochastic Processes and their Applications} \textbf{106}, 107--139.

\end{thebibliography}
\end{document}